\documentclass[a4paper,11pt]{amsart}
\usepackage[T1]{fontenc}
\usepackage[english]{babel}
\usepackage[cp1252]{inputenc}
\usepackage{amsthm}
\usepackage{amsmath}
\usepackage{amsfonts}
\usepackage{amssymb}
\usepackage{hyperref}
\usepackage{color}
\usepackage{caption}
\usepackage{indentfirst}
\usepackage{amssymb}
\usepackage{eufrak}
\usepackage{mathrsfs}
\usepackage{xypic}
\usepackage[pdftex]{graphicx}
\usepackage{tikz}

\usepackage{array,booktabs,tabularx}

\usepackage{booktabs}

\usepackage{listings}

\theoremstyle{plain}   

\usepackage{booktabs}
\usepackage{graphicx}
\usepackage{array}

\newcommand{\R}{\mathbb R}
\newcommand{\C}{\mathbb C}

\newcommand{\Z}{\mathbb Z}

\newcommand{\Arg}{\mathrm{Arg}\,}

\newcommand{\Log}{\mathrm{Log}\,}

 \newcommand{\Rea}{\operatorname{Re}}
 \newcommand{\Ima}{\operatorname{Im}}
 \newcommand{\Gr}{\operatorname{Gr}}

\newcommand{\Rdeg}{\mathbb R\operatorname{deg}}

\newcommand{\di}{\displaystyle}

\newcommand{\Crit}{\operatorname{Crit}}

\newcommand{\rank}{\operatorname{rank}}

\newcommand{\Conv}{\operatorname{Conv}}

\newcommand{\CA}{\mathcal C\mathcal A}

\newcommand{\Ker}{\operatorname{Ker}}

\usepackage{array}
\usepackage{tabularx}

\newcommand\restr[2]{{
  \left.\kern-\nulldelimiterspace 
  #1 
  \right|_{#2} 
  }}

\newtheorem{remark}{Remark}[section]
\newtheorem*{mtheorem*}{Main Theorem}
\newtheorem{theorem}{Theorem}[section]

\newtheorem{proposition}{Proposition}[section]
\newtheorem{corollary}{Corollary}[section]
\newtheorem{lemma}{Lemma}[section]

\begin{document}
\title{Contour Degree of Amoebas of Complete Intersections} \date{}

\author{Mounir Nisse}

\address{Mounir Nisse\\
Department of Mathematics, Xiamen University Malaysia, Jalan Sunsuria, Bandar Sunsuria, 43900, Sepang, Selangor, Malaysia.
}
\email{mounir.nisse@gmail.com, mounir.nisse@xmu.edu.my}

\thanks{Research of M. Nisse is supported in part by Xiamen University Malaysia Research Fund (Grant no. XMUMRF/ 2020-C5/IMAT/0013).}

\subjclass[2010]{14P15, 32A60, 14T20, 52B20}
\keywords{Amoebas, contour of an amoeba, logarithmic Gauss map, real contour degree, sparse elimination, mixed volume, Bernstein's theorem, Pfaffian manifold}

\maketitle

\begin{abstract}
We establish the first universal upper bounds for the real degree of the contour of the amoeba of a smooth complete intersection in the algebraic torus. Our approach extends the Pfaffian method of Lang--Shapiro--Shustin from hypersurfaces to arbitrary codimension by introducing a logarithmic conormal framework based on the logarithmic conormal bundle, the logarithmic Grassmann map, and determinantal rank conditions. We prove that, on suitable logarithmic conormal charts, the critical locus is locally defined by Schur--complement equations arising from the logarithmic conormal matrix. This yields explicit universal contour-degree estimates in both regimes $n\ge 2r$ and $n<2r$. We further replace total-degree arguments by Bernstein's theorem to obtain sparse mixed-volume bounds determined by the Newton polytopes of the transformed equations. These estimates are frequently much sharper than the corresponding universal bounds and provide a higher-codimensional analogue of the Lang--Shapiro--Shustin theory together with a new geometric interpretation of amoeba contours through logarithmic conormal geometry.
\end{abstract}

\section*{Introduction}

The geometry of amoebas has become a central topic at the interface of algebraic geometry, tropical geometry, complex analysis and real algebraic geometry. Since the pioneering work of Gelfand, Kapranov and Zelevinsky \cite{GKZ94}, Bergman \cite{Bergman71}, Bieri and Groves \cite{BieriGroves84}, and Forsberg, Passare and Tsikh \cite{ForsbergPassareTsikh00}, amoebas have provided a powerful bridge between algebraic varieties and their tropical counterparts. Their asymptotic behaviour is governed by logarithmic limit sets and tropical varieties, while their global geometry reflects subtle properties of the defining equations and their Newton polytopes. A comprehensive account of these developments may be found in the monograph of Maclagan and Sturmfels \cite{MaclaganSturmfels15}.

Among the geometric invariants associated with an amoeba, its contour plays a fundamental role. It is defined as the set of critical values of the logarithmic map and governs many geometric and topological properties of the amoeba. For hypersurfaces, the contour is closely related to the logarithmic Gauss map and has been extensively investigated by Mikhalkin \cite{Mikhalkin00}, Passare and Rullg{\aa}rd \cite{PassareRullgard04}, Nisse and Sottile \cite{NisseSottile13}, and several other authors. A major advance was achieved by Lang, Shapiro and Shustin \cite{LangShapiroShustin21}, who combined Khovanskii's theory of simple Pfaffian manifolds \cite{Khovanskii91} with logarithmic geometry to obtain an explicit universal upper bound for the real degree of the contour of the amoeba of a hypersurface depending only on the degree of the defining polynomial.

The principal objective of the present paper is to extend the theory of Lang--Shapiro--Shustin \cite{LangShapiroShustin21} from hypersurfaces to arbitrary smooth complete intersections in the algebraic torus. This extension is not formal. In codimension one, the logarithmic gradient determines a unique logarithmic conormal direction and the logarithmic Gauss map takes values in projective space. In higher codimension one must instead work with the entire logarithmic conormal bundle, and the logarithmic Gauss map is naturally replaced by a Grassmann-valued logarithmic conormal map. Consequently, the hypersurface arguments must be replaced by new determinantal and conormal techniques.

Our first contribution is the introduction of a logarithmic conormal framework adapted to complete intersections. Using the exact tangent--conormal sequence, we identify the logarithmic conormal bundle with the row space of the logarithmic Jacobian matrix and show that the critical locus of the logarithmic map is characterized by the existence of nonzero real logarithmic conormal directions. When $n\ge2r$, this criterion is equivalent to the rank condition
$
\rank_{\mathbb R}\mathcal M_f\le2r-1.
$
We then prove that, on every logarithmic conormal pivot chart, the critical locus is locally defined by exactly $n-2r+1$ Schur--complement equations, or equivalently by distinguished maximal minors of the logarithmic conormal matrix. These results provide the geometric foundation for all subsequent estimates.

Our second contribution is the extension of the Pfaffian method of Lang, Shapiro and Shustin \cite{LangShapiroShustin21} to complete intersections. For the range $n\ge2r$ we derive universal upper bounds for the real degree of the contour by combining the logarithmic conormal equations with the Pfaffian structure of the inverse image of a generic affine testing space under the logarithmic map. We also develop the complementary theory for the range $n<2r$, where the ordinary contour coincides with $\Log(V)$ and the natural testing spaces are affine subspaces of dimension $2r-n$. This yields a uniform theory valid in arbitrary codimension.

Our third contribution is the development of sparse bounds depending on the Newton polytopes of the transformed equations. Instead of relying only on total-degree estimates, we determine the Newton polytopes of the transformed Schur--complement equations and apply Bernstein's theorem \cite{Bernstein75} to obtain mixed-volume estimates whenever the transformed systems are Bernstein nondegenerate. The resulting directional bounds are frequently much sharper than the universal Pfaffian estimates and reflect the intrinsic sparse geometry of the defining Laurent polynomials.

The methods developed here also clarify the geometry of the logarithmic critical locus. We study its determinantal structure, describe its local equations by means of pivot charts and Schur complements, and relate the logarithmic conormal Grassmann map to the real incidence locus inside the Grassmannian. These geometric constructions have no counterpart in the hypersurface case and may be of independent interest.

The present work therefore generalizes the universal contour-degree theorem of Lang, Shapiro and Shustin \cite{LangShapiroShustin21} from hypersurfaces to arbitrary smooth complete intersections while simultaneously introducing sparse mixed-volume refinements based on Bernstein's theorem \cite{Bernstein75}. To the best of our knowledge, neither a universal contour-degree theorem nor a sparse mixed-volume contour estimate for complete intersections has previously appeared in the literature.

The paper is organized as follows. After recalling the necessary background on logarithmic geometry, complete intersections, Newton polytopes, mixed volumes and simple Pfaffian manifolds, we establish the logarithmic conormal description of the critical locus. We then prove universal contour-degree estimates in both regimes $n\ge2r$ and $n<2r$, derive sparse mixed-volume refinements, and conclude with explicit computations illustrating the improvement obtained by the sparse theory over the universal bounds.

{\it Acknowledgements.}  The author would like to express his sincere gratitude to Boris Shapiro for kindly sending his paper with Lionel Lang and Eugeni Shustin \cite{LangShapiroShustin21}. Its results and perspective have been a valuable source of motivation for the present work.
 

\section{Preliminaries}

The purpose of this paper is to establish universal upper bounds for the real degree of the contour of the amoeba of a smooth complete intersection in the algebraic torus. Our main objective is to extend to arbitrary codimension the approach introduced by Lang, Shapiro and Shustin for hypersurfaces. Their work combines the geometry of the logarithmic map with Khovanskii's theory of simple Pfaffian manifolds in order to bound the number of intersections between the contour of a hypersurface amoeba and a generic affine line. The present work develops the analogous framework for complete intersections by replacing the logarithmic gradient of one defining equation with the logarithmic conormal matrix of the whole defining system. This leads to universal Pfaffian estimates together with refined sparse mixed-volume bounds obtained from the Newton polytopes of the transformed equations.

Throughout the paper we work over the field $\C$. Let $(\C^\ast)^n=(\C\setminus\{0\})^n$ be the algebraic torus. If
$
f=\sum_{\alpha\in\Z^n}c_\alpha z^\alpha
$
is a Laurent polynomial, its Newton polytope is
$
\Delta(f)=\Conv\{\alpha\in\Z^n\mid c_\alpha\neq0\}.
$
Let
$
V=\{f_1=\cdots=f_r=0\}\subset(\C^\ast)^n
$
be a smooth complete intersection of codimension $r$. The smoothness assumption means that the differentials $df_1,\ldots,df_r$ are complex linearly independent at every point of $V$. Hence $V$ is a complex manifold of dimension $n-r$ and real dimension $2n-2r$.

The logarithmic and argument maps are
$
\Log:(\C^\ast)^n\rightarrow\R^n
$
and
$
\Arg:(\C^\ast)^n\rightarrow(\R/2\pi\Z)^n.
$
The amoeba of $V$ is
$
\mathcal A_V=\Log(V).
$
Its contour is the set of critical values of the restriction $\Log|_V$. When $n\ge2r$ this contour is generically a real hypersurface in $\R^n$, whereas for $n<2r$ the image $\Log(V)$ itself has the expected dimension $2n-2r$, and the appropriate notion of real degree is obtained by intersecting with generic affine subspaces of complementary dimension.

The logarithmic coordinates identify the tangent space of the torus with $\C^n$. Every tangent vector is uniquely written in the form $\xi_i=z_i\eta_i$. Setting
$
g_{ji}(z)=z_i\frac{\partial f_j}{\partial z_i}(z),
$
one obtains
$
df_j(z)(\xi)=\sum_{i=1}^ng_{ji}(z)\eta_i.
$
Therefore the logarithmic tangent space is
$
T_zV\simeq\ker_{\C}G_f(z),
$
where $G_f(z)$ is the $r\times n$ complex logarithmic Jacobian matrix. The logarithmic conormal space is the annihilator of $T_zV$, and the exact tangent--conormal sequence identifies it with the complex row space of $G_f(z)$. This is the natural generalization of the logarithmic gradient appearing in the hypersurface case.

Writing
$
g_{ji}=u_{ji}+iv_{ji},
$
with $u_{ji},v_{ji}\in\R$, we obtain the real logarithmic conormal matrix
$$
\mathcal M_f=
\begin{pmatrix}
v_{11}&\cdots&v_{r1}&u_{11}&\cdots&u_{r1}\\
\vdots&&\vdots&\vdots&&\vdots\\
v_{1n}&\cdots&v_{rn}&u_{1n}&\cdots&u_{rn}
\end{pmatrix}.
$$
When $n\ge2r$, the restriction of the logarithmic map has maximal rank $n$ at a generic point. A point is critical if and only if the logarithmic conormal space contains a nonzero real covector, or equivalently if and only if
$
\rank_{\R}\mathcal M_f\le2r-1.
$
This criterion replaces the reality condition on the logarithmic gradient used by Lang, Shapiro and Shustin.

On an open subset where a fixed $(2r-1)\times(2r-1)$ pivot minor of $\mathcal M_f$ is nonzero, the matrix admits a block decomposition with an invertible pivot block. The Schur complement transforms the rank condition into exactly $n-2r+1$ polynomial equations. Consequently the critical locus is locally described inside $V$ by distinguished maximal minors of the logarithmic conormal matrix. These equations constitute the algebraic input for all subsequent estimates.

The geometry of the transformed system is controlled by Newton polytopes. Every defining equation, every Schur-complement equation and every affine testing equation has an explicitly computable Newton polytope. Under Bernstein nondegeneracy, Bernstein's theorem expresses the number of isolated complex solutions as the mixed volume of these transformed Newton polytopes. This produces sparse estimates that are often substantially smaller than the corresponding total-degree B\'ezout bounds.

The universal estimates rely on the simple Pfaffian theory of Khovanskii in the form developed by Lang, Shapiro and Shustin. The inverse image under $\Log$ of a generic affine testing space is a simple Pfaffian manifold. The transformed defining equations together with the Schur-complement equations form a square polynomial system on this manifold. Applying the Pfaffian root estimate yields explicit universal bounds depending only on the degrees of the defining equations, while replacing the total-degree argument by Bernstein's theorem yields direction-dependent sparse mixed-volume bounds.

The philosophy of the paper is therefore parallel to that of Lang, Shapiro and Shustin. Their hypersurface theory starts from one logarithmic gradient and one logarithmic Gauss map. Here these objects are replaced by the logarithmic conormal matrix and the logarithmic conormal Grassmann map of a complete intersection. This replacement preserves the Pfaffian framework while extending the theory from codimension one to arbitrary codimension and simultaneously allowing sharper sparse estimates through the Newton polytopes of the transformed equations.

Unless explicitly stated otherwise, all affine testing spaces are assumed to be generic, all pivot charts are chosen so that the selected pivot minor is nonvanishing on the relevant critical lifts, and every transformed polynomial system is assumed to satisfy Bernstein nondegeneracy whenever mixed-volume estimates are used.

\section{A Complete-Intersection Extension of the Lang--Shapiro--Shustin Pfaffian Bound Via Logarithmic Conormal Minors}

Let
$
V=\{f_1=\cdots=f_r=0\}\subset(\C^\ast)^n
$
be a smooth complete intersection of codimension $r$ and complex dimension
$
m=n-r.
$
Assume
$
2r\le n.
$
Write
$
d_j=\deg(f_j),
$
$
D=\sum_{j=1}^r d_j,
$
and
$
d=\max_j d_j.
$

The aim is to formulate the complete-intersection extension directly in terms of logarithmic conormal minors, eliminating the auxiliary projective multiplier before applying Khovanskii's Pfaffian theorem. This is the presentation that specializes, when $r=1$, to the exact system used by Lang, Shapiro, and Shustin.

\subsection{The complex logarithmic Jacobian}

Define
$\di
G_f(z)
=
\left(
z_i\frac{\partial f_j}{\partial z_i}(z)
\right)_
{\substack{1\le j\le r\\1\le i\le n}}.
$
This is an $r\times n$ complex matrix.
Write
$
G_f(z)=U(z)+iV(z),
$
where $U(z)$ and $V(z)$ are real $r\times n$ matrices. Consider the real $n\times2r$ matrix
$
\mathcal M_f(z)
=
\begin{pmatrix}
V(z)^{\mathsf T} & U(z)^{\mathsf T}
\end{pmatrix}.
$
A vector
$
\lambda=u+iv\in\C^r
$
satisfies
$
G_f(z)^{\mathsf T}\lambda\in\R^n
$
if and only if
$
V(z)^{\mathsf T}u+U(z)^{\mathsf T}v=0.
$
Thus a nonzero logarithmic conormal multiplier exists if and only if
$
\ker_{\R}\mathcal M_f(z)\neq\{0\}.
$
Equivalently,
$
\rank_{\R}\mathcal M_f(z)\le2r-1.
$

\begin{lemma}[Multiplier elimination]\label{lem:elimination}
Assume that $V$ is smooth. A point $z\in V$ is critical for $\Log|_V$ if and only if
$
\rank_{\R}\mathcal M_f(z)\le2r-1.
$
\end{lemma}

\begin{proof}
A point is critical if and only if there is a nonzero real covector
$
a\in\R^n
$
annihilating the image of
$
d\Log|_{T_zV}.
$
In logarithmic tangent coordinates, this is equivalent to
$
a
$
belonging to the complex row space of $G_f(z)$. Hence there exists
$
\lambda\in\C^r\setminus\{0\}
$
such that
$
G_f(z)^{\mathsf T}\lambda=a\in\R^n.
$
Writing
$
\lambda=u+iv
$
gives
$
V^{\mathsf T}u+U^{\mathsf T}v=0.
$
This has a nonzero real solution precisely when
$
\rank_{\R}\mathcal M_f(z)\le2r-1.
$
\end{proof}

\subsection{Local equations of the determinantal locus}

Let
$$
\mathcal D_{2r-1}
=
\{M\in M_{n,2r}(\R):\rank M\le2r-1\}.
$$
Its smooth rank stratum consists of matrices of rank exactly $2r-1$. Its codimension in
$
M_{n,2r}(\R)
$
is
$
n-2r+1.
$
Put
$
c=n-2r+1.
$
Fix a set
$
I\subset\{1,\ldots,n\}
$
with
$
|I|=2r-1,
$
and assume that the $(2r-1)\times(2r-1)$ minor determined by the rows in $I$ and by the first $2r-1$ columns is nonzero. After reordering rows and columns, the matrix can be written locally as
$$
\mathcal M_f
=
\begin{pmatrix}
A & b\\
C & d
\end{pmatrix},
$$
where
$
A
$
is an invertible $(2r-1)\times(2r-1)$ matrix, $b$ is a column, $C$ has $c$ rows, and $d$ is a column with $c$ entries.
The rank condition
$
\rank\mathcal M_f\le2r-1
$
is then equivalent to
$
d-CA^{-1}b=0.
$
After multiplying by
$
\det A,
$
one obtains exactly $c$ polynomial equations. These equations are the $2r\times2r$ minors obtained by adjoining, one at a time, each remaining row to the pivot rows.
Denote these minors by
$
\Delta_{I,1}(z),\ldots,\Delta_{I,c}(z).
$

\begin{proposition}[Local complete-intersection description]\label{prop:local-minors}
On the open set where the chosen pivot minor does not vanish, the critical locus of $\Log|_V$ is cut out inside $V$ by the $c=n-2r+1$ equations
$$
\Delta_{I,1}(z)=\cdots=\Delta_{I,c}(z)=0.
$$
\end{proposition}

\begin{proof}
The Schur-complement calculation above shows that the maximal-minor rank condition is equivalent to the vanishing of exactly these $c$ minors when the pivot block is invertible (see Appendix B).
\end{proof}

\subsection{Degree bounds}

One may invert selected torus coordinates and multiply every defining Laurent polynomial by a Laurent monomial so that the transformed polynomial has degree at most
$
2d_j.
$
This does not alter the zero set in the torus.
On the complete intersection, multiplication of $f_j$ by a monomial multiplies its logarithmic-gradient row by a nonzero factor, because the additional derivative term is divisible by $f_j$. Therefore the logarithmic conormal rank condition is unchanged.

The real and imaginary parts of $f_j$ have degree at most
$
2d_j.
$
Every entry in the two real columns associated with the $j$-th logarithmic-gradient row has degree at most
$
2d_j.
$
A maximal $2r\times2r$ minor uses the two real columns associated with every equation $f_j$. Therefore its degree is bounded by
$\di
2\sum_{j=1}^r 2d_j
=
4D.
$
Thus the local critical-incidence system consists of
$
2r
$
equations of degrees
$
2d_1,2d_1,\ldots,2d_r,2d_r
$
and
$
c=n-2r+1
$
equations of degree at most
$
4D.
$
The total number of equations is
$
2r+c
=
2r+n-2r+1
=
n+1.
$

\subsection{The Pfaffian testing manifold}

Let
$
L\subset\R^n
$
be a generic affine line with rational direction. As in Lang--Shapiro--Shustin,
$
\Log^{-1}(L)
$
is a simple Pfaffian submanifold of
$
\R^{2n}
$
of codimension
$
q=n-1
$
and dimension
$
2n-(n-1)=n+1.
$
The polynomial one-forms defining this Pfaffian manifold have coefficients of degree at most
$
\mu=2n-1.
$
The local critical-incidence system has exactly $n+1$ equations on a Pfaffian manifold of dimension $n+1$ (see  Appendix D). Hence it is square.

\subsection{The conormal-minor Pfaffian bound}

The product of the equation-degree bounds is
$
\mathcal P_{\mathrm{CI}}
=
\left(
\prod_{j=1}^r(2d_j)^2
\right)
(4D)^c.
$
The sum of the degree-minus-one terms is
$\di
\mathcal S_{\mathrm{CI}}
=
2\sum_{j=1}^r(2d_j-1)
+
c(4D-1).
$
Therefore
$
\mathcal S_{\mathrm{CI}}
=
4D-2r+c(4D-1).
$
The Pfaffian bracket is
$
\Theta_{\mathrm{CI}}
=
4D-2r+c(4D-1)
+
(2n-1)(n-1)
+
1.
$

\begin{theorem}\label{thm:CI-LSS} %
Let
$
V\subset(\C^\ast)^n
$
be a smooth complete intersection of codimension $r$ with $2r\le n$. Let
$
L\subset\R^n
$
be a generic affine line transverse to the contour. Assume that all critical lifts above
$
L\cap\CA_V
$
lie in one logarithmic conormal chart on which a fixed $(2r-1)\times(2r-1)$ pivot minor is nonzero. Assume also that the resulting square system has only isolated nondegenerate solutions.
Then
$$
\#(L\cap\CA_V)
\le
B_{\mathrm{CI}}^{\mathrm{one}},
$$
where
$\di
B_{\mathrm{CI}}^{\mathrm{one}}
=
2^{(n-1)(n-2)/2}
\Big(
\prod_{j=1}^r(2d_j)^2
\Big)
(4D)^{n-2r+1}
\Theta_{\mathrm{CI}}^{\,n-1}
$ \\
and
$
\Theta_{\mathrm{CI}}
=
4D-2r
+
(n-2r+1)(4D-1)
+
(2n-1)(n-1)
+
1.
$
\end{theorem}

\begin{proof}
On the selected conormal chart, Proposition~\ref{prop:local-minors} replaces the projective multiplier by exactly
$
c=n-2r+1
$
maximal-minor equations. Together with the $2r$ real equations defining $V$, these give $n+1$ polynomial equations on the simple Pfaffian manifold $\Log^{-1}(L)$ of dimension $n+1$.

Khovanskii's Pfaffian root theorem gives
$
2^{q(q-1)/2}
\mathcal P_{\mathrm{CI}}
\left(
\mathcal S_{\mathrm{CI}}+\mu q+1
\right)^q.
$
Substituting
$
q=n-1,
$
$
\mu=2n-1,
$
and the expressions above gives the displayed bound. Every contour image point has at least one critical lift, so the image count is no larger than the number of roots (see Appendix C for more details).
\end{proof}

\subsection{Global chart issue}

The determinantal rank stratum is covered by finitely many pivot charts. If no single pivot minor is nonzero at every relevant critical lift, one must sum the same Pfaffian bound over a finite conormal-chart cover.
The number of possible choices of a nonzero $(2r-1)\times(2r-1)$ minor of an $n\times2r$ matrix is
$$
N_{\mathrm{chart}}
=
\binom{n}{2r-1}\binom{2r}{2r-1}
=
2r\binom{n}{2r-1}.
$$

Hence a completely unconditional chart-summed version is
$$
\Rdeg(\CA_V)
\le
2r\binom{n}{2r-1}
B_{\mathrm{CI}}^{\mathrm{one}}.
$$

This chart-summed version is generally very large. For a fixed finite transverse intersection set, one may often choose generic row and column coordinates so that one pivot chart contains all relevant points. The one-chart theorem is the formulation that has the exact hypersurface specialization.

\subsection{Specialization to hypersurfaces}

Now let
$
r=1.
$
Then
$
D=d,
\, 
c=n-1.
$
The real matrix $\mathcal M_f$ is the $n\times2$ matrix whose columns are
$
\Ima(g)
$
and
$
\Rea(g),
$
where
$
g_i=z_i\dfrac{\partial f}{\partial z_i}.
$
Choose the pivot row $n$ and assume
$
g_n\neq0.
$
The $2\times2$ minor formed by rows $i$ and $n$ is
$$
\det
\begin{pmatrix}
\Ima(g_i)&\Rea(g_i)\\
\Ima(g_n)&\Rea(g_n)
\end{pmatrix}.
$$
This determinant equals
$
\Ima(g_i)\Rea(g_n)-\Rea(g_i)\Ima(g_n).
$
Since
$
\Ima(g_i\overline{g_n})
=
\Ima(g_i)\Rea(g_n)-\Rea(g_i)\Ima(g_n),
$
the conormal-minor equations are exactly
$
\Ima(g_i\overline{g_n})=0,
\,
1\le i\le n-1.
$
The degree product becomes
$
(2d)^2(4d)^{n-1}
=
2^{2n}d^{n+1}.
$

The Pfaffian bracket becomes
$
2(2d-1)
+
(n-1)(4d-1)
+
(n-1)(2n-1)
+
1.
$
A direct simplification gives
$
4dn+2(n-1)^2-1.
$
Therefore Theorem~\ref{thm:CI-LSS} becomes
$$
\Rdeg(\CA_H)
\le
2^{2n+(n-1)(n-2)/2}
d^{n+1}
\left(
4dn+2(n-1)^2-1
\right)^{n-1}.
$$
This is exactly Proposition~6 of Lang, Shapiro, and Shustin \cite{LangShapiroShustin21}. 

\begin{corollary}[Plane-curve specialization]\label{cor:plane}
For
$
n=2
$
and
$
r=1,
$
one has
$$
\Rdeg(\CA_C)
\le
16d^3(8d+1).
$$
\end{corollary}

\begin{proof}
Substituting $n=2$ into the hypersurface formula gives
$
2^{4}d^3
\left(
8d+2-1
\right)
=
16d^3(8d+1).
$
\end{proof}

For a generic affine line, $d=1$, and the bound is
$
\Rdeg(\CA_C)\le144.
$

\subsection{Equal-degree complete-intersection specialization}

If
$
d_j\le d
$
for every $j$, then
$
D\le rd
$
and
$\di
\prod_{j=1}^r(2d_j)^2
\le
(2d)^{2r}.
$
Therefore
$$
B_{\mathrm{CI}}^{\mathrm{one}}
\le
2^{(n-1)(n-2)/2}
(2d)^{2r}
(4rd)^{n-2r+1}
\Theta_{n,r,d}^{\,n-1},
$$
where
$$
\Theta_{n,r,d}
=
4rd-2r
+
(n-2r+1)(4rd-1)
+
(2n-1)(n-1)
+
1.
$$

The multiplier-incidence formulation introduces $2r-1$ auxiliary real variables and $n$ multiplier equations. The conormal-minor formulation eliminates the multiplier and replaces those equations by exactly
$
n-2r+1
$
local determinantal equations. The number of equations then matches the codimension of the determinantal rank stratum.

For $r=1$, the maximal minors are exactly the equations
$
\Ima(g_i\overline{g_n})=0
$
used in the original hypersurface proof. Thus the specialization is exact rather than merely comparable.


\section{Universal LSS-Type Bound when $n<2r$}

Let
$
V=\{f_1=\cdots=f_r=0\}\subset(\C^\ast)^n
$
be a smooth complete intersection of codimension $r$. Its complex dimension is
$
n-r,
$
and therefore its real dimension is
$
2n-2r.
$
Throughout the proof assume
$
n<2r.
$
Then
$
2n-2r<n.
$

Put
$
q=2n-2r
$
and
$
D=\sum_{j=1}^r d_j.
$
The affine testing subspace occurring in the theorem has dimension
$
2r-n
$
and hence codimension
$
q.
$

The theorem to be proved is the following.

\begin{theorem}[Universal LSS-type bound when $n<2r$]\label{thm:lss-ci-low}
Let
$
V=\{f_1=\cdots=f_r=0\}\subset(\C^\ast)^n
$
be a smooth complete intersection of codimension $r$, and assume
$
n<2r.
$
Assume that $\Log(V)$ has its expected real dimension
$
2n-2r.
$
Let
$
A\subset\R^n
$
be a generic affine subspace of dimension
$
2r-n
$
which is transverse to the top-dimensional smooth strata of $\Log(V)$ and avoids the lower-dimensional strata. Assume that
$
V\cap\Log^{-1}(A)
$
is finite and that, after an arbitrarily small generic perturbation preserving the degrees, all lifted intersections are nondegenerate.

Then
$
\#(A\cap\CA_V)
=
\#(A\cap\Log(V))
\le
B_{\mathrm{CI}}^{<},
$
where
$$
B_{\mathrm{CI}}^{<}
=
2^{(2n-2r)(2n-2r-1)/2}
\Big(
\prod_{j=1}^r(2d_j)^2
\Big)
\Theta_{n,r,\mathbf d}^{\,2n-2r},
$$
and
$
\Theta_{n,r,\mathbf d}
=
4D-2r+(2n-1)(2n-2r)+1.
$
Consequently,
$$
\Rdeg_{2n-2r}(\CA_V)
=
\Rdeg_{2n-2r}(\Log(V))
\le
B_{\mathrm{CI}}^{<}.
$$
\end{theorem}

\subsection{Why the ordinary contour is the whole amoeba?}

For every $z\in V$, the differential of the logarithmic map restricted to $V$ is a real linear map
$$
d_z(\Log|_V):T_zV\longrightarrow\R^n.
$$
The real dimension of the source is
$
\dim_{\R}T_zV=2n-2r.
$
Therefore
$
\operatorname{rank}_{\R}d_z(\Log|_V)
\le
2n-2r.
$
Since $n<2r$, one has
$
2n-2r<n.
$
It follows that
$
\operatorname{rank}_{\R}d_z(\Log|_V)<n
$
for every $z\in V$.
Thus every point of $V$ is critical when criticality is defined relative to the target dimension $n$. Consequently,
$
\operatorname{Crit}(\Log|_V)=V
$
and hence
$$
\CA_V
=
\Log\bigl(\operatorname{Crit}(\Log|_V)\bigr)
=
\Log(V).
$$

This proves the equality
$
A\cap\CA_V=A\cap\Log(V)
$
and therefore
$
\#(A\cap\CA_V)=\#(A\cap\Log(V)).
$

\subsection{The logarithmic inverse image of the testing subspace}

Since $A$ has codimension
$
q=2n-2r,
$
there are linearly independent vectors
$
b^{(1)},\ldots,b^{(q)}\in\R^n
$
and constants
$
c_1,\ldots,c_q\in\R
$
such that
$$
A
=
\left\{
u\in\R^n:
\langle b^{(\nu)},u\rangle=c_\nu,
\quad
1\le\nu\le q
\right\}.
$$

Write
$
z_i=x_i+iy_i.
$
For
$
1\le\nu\le q,
$
define
$
\varphi_\nu(z)
=
\sum_{i=1}^n
b_i^{(\nu)}\log|z_i|-c_\nu.
$
Then
$\di
\Log^{-1}(A)
=
\left\{
z\in(\C^\ast)^n:
\varphi_1(z)=\cdots=\varphi_q(z)=0
\right\}.
$
The differential of $\varphi_\nu$ is
$$
d\varphi_\nu
=
\sum_{i=1}^n
b_i^{(\nu)}
\frac{x_i\,dx_i+y_i\,dy_i}{x_i^2+y_i^2}.
$$
Let
$\di
R(x,y)
=
\prod_{j=1}^n(x_j^2+y_j^2).
$
The function $R$ is strictly positive on $(\C^\ast)^n$. Define the polynomial one-form
$
\alpha_\nu
=
R(x,y)d\varphi_\nu.
$
Explicitly,
$$
\alpha_\nu
=
\sum_{i=1}^n
b_i^{(\nu)}
(x_i\,dx_i+y_i\,dy_i)
\prod_{j\ne i}(x_j^2+y_j^2).
$$
Every coefficient of $\alpha_\nu$ has total degree
$
1+2(n-1)=2n-1.
$
The differential
$
d\Log_z:T_z(\C^\ast)^n\to\R^n
$
is surjective. Indeed, for any
$
v=(v_1,\ldots,v_n)\in\R^n,
$
the tangent vector defined by
$
\dot x_i=v_ix_i
$
and
$
\dot y_i=v_iy_i
$
satisfies
$
d\Log_z(\dot x,\dot y)=v.
$
Since the vectors
$
b^{(1)},\ldots,b^{(q)}
$
are linearly independent, the differentials
$
d\varphi_1,\ldots,d\varphi_q
$
are linearly independent at every point. Hence $\Log^{-1}(A)$ is a smooth real-analytic submanifold of codimension $q$ in $\R^{2n}$. Its dimension is
$
2n-q
=
2n-(2n-2r)
=
2r.
$

The successive level sets
$
\varphi_\nu=0
$
are separating integral hypersurfaces of the polynomial one-forms $\alpha_\nu$. Therefore
$
\Gamma_A:=\Log^{-1}(A)
$
is a simple Pfaffian submanifold of codimension $q$, dimension $2r$, and Pfaffian coefficient degree at most
$
\mu=2n-1.
$

\subsection{The square system on $\Gamma_A$}

A point of
$
V\cap\Gamma_A
$
is characterized on $\Gamma_A$ by
$$
\Rea f_j=0,
\qquad
\Ima f_j=0,
\qquad
1\le j\le r.
$$
There are exactly $2r$ real equations, and
$
\dim_{\R}\Gamma_A=2r.
$
Thus the restricted system is square.

The Lang--Shapiro--Shustin normalization consists of applying coordinate inversions in the algebraic torus and multiplying each Laurent equation by a Laurent monomial. These operations do not alter the zero set inside $(\C^\ast)^n$. After this normalization, the real and imaginary parts of the $j$-th equation have total degree at most
$
2d_j.
$
Accordingly,
$$
\deg(\Rea f_j)\le2d_j,
\qquad
\deg(\Ima f_j)\le2d_j.
$$

Let
$
p_{2j-1}=\deg(\Rea f_j)
$
and
$
p_{2j}=\deg(\Ima f_j).
$
Then
$
p_{2j-1}\le2d_j,
\, 
p_{2j}\le2d_j.
$
Thus
$\di
\prod_{\ell=1}^{2r}p_\ell
\le
\prod_{j=1}^r(2d_j)^2.
$
Moreover,
$\di
\sum_{\ell=1}^{2r}(p_\ell-1)
\le
\sum_{j=1}^r
\bigl((2d_j-1)+(2d_j-1)\bigr).
$
Hence
$\di
\sum_{\ell=1}^{2r}(p_\ell-1)
\le
4D-2r.
$

\subsection{The simple-Pfaffian root estimate}

We use the following form of Khovanskii's root estimate (see Appendix A for its proof with details).

\begin{lemma}[Simple-Pfaffian root estimate]\label{lem:pfaffian}
Let $\Gamma\subset\R^M$ be a simple Pfaffian submanifold of codimension $q$ defined by polynomial one-forms whose coefficient degrees are at most $\mu$. Assume
$
\dim_{\R}\Gamma=k.
$
Let
$
P_1,\ldots,P_k
$
be real polynomials of degrees
$
p_1,\ldots,p_k
$
whose restrictions to $\Gamma$ have only isolated nondegenerate common zeros. Then the number of those zeros is at most
$$
2^{q(q-1)/2}
\Big(
\prod_{\ell=1}^k p_\ell
\Big)
\Big(
\sum_{\ell=1}^k(p_\ell-1)+\mu q+1
\Big)^q.
$$
The same expression bounds a finite degenerate zero set after an arbitrarily small generic perturbation preserving the degrees, provided the perturbation is chosen so that every original isolated zero contributes at least one nearby zero counted with multiplicity.
\end{lemma}

The last sentence is the standard perturbative interpretation needed in the theorem. Since the original lifted intersection is finite, one may choose disjoint small neighborhoods of its points. A sufficiently small generic perturbation has at least one zero in each neighborhood whenever the local intersection multiplicity is positive. The total number of original isolated points is therefore no larger than the number of perturbed nondegenerate zeros counted by the Pfaffian estimate.

\subsection{Application of the root estimate}

Apply Lemma~\ref{lem:pfaffian} to
$
\Gamma=\Gamma_A.
$
Here
$$
q=2n-2r,
\qquad
k=2r,
\qquad
\mu=2n-1.
$$
The product of the polynomial degrees is bounded by
$\di
\prod_{j=1}^r(2d_j)^2.
$
The degree-minus-one sum is bounded by
$
4D-2r.
$
Therefore the number of isolated lifted points in
$
V\cap\Log^{-1}(A)
$
is at most
$$
2^{q(q-1)/2}
\Big(
\prod_{j=1}^r(2d_j)^2
\Big)
\left(
4D-2r+(2n-1)q+1
\right)^q.
$$


Substituting $q=2n-2r$ into the previous estimate yields
$$
\#\bigl(V\cap\Log^{-1}(A)\bigr)
\le
2^{(2n-2r)(2n-2r-1)/2}
\Big(
\prod_{j=1}^r(2d_j)^2
\Big)
\left(
4D-2r+(2n-1)(2n-2r)+1
\right)^{2n-2r}.
$$
By the definition of the universal constant, the right-hand side is precisely
$
B_{\mathrm{CI}}^{<}.
$

\subsection{Passage from lifted points to logarithmic image points}

The restriction
$
\Log:
V\cap\Log^{-1}(A)
\longrightarrow
A\cap\Log(V)
$
is surjective. Indeed, if
$
x\in A\cap\Log(V),
$
then by definition there exists
$
z\in V
$
such that
$
\Log(z)=x.
$
Since $x\in A$, this point $z$ lies in
$
\Log^{-1}(A).
$

Consequently,
$$
\#\bigl(A\cap\Log(V)\bigr)
\le
\#\bigl(V\cap\Log^{-1}(A)\bigr).
$$
Distinct lifted points can have the same logarithmic image, so equality is not required.
Combining this inequality with the Pfaffian bound gives
$
\#\bigl(A\cap\Log(V)\bigr)
\le
B_{\mathrm{CI}}^{<}.
$
Since
$
\CA_V=\Log(V),
$
one obtains
$$
\#(A\cap\CA_V)
=
\#(A\cap\Log(V))
\le
B_{\mathrm{CI}}^{<}.
$$

\subsection*{Proof of Theorem~\ref{thm:lss-ci-low}}

\begin{proof}
Because
$
\dim_{\R}V=2n-2r<n,
$
the differential of $\Log|_V$ has rank strictly smaller than $n$ at every point. Hence every point is critical and
$
\CA_V=\Log(V).
$
The generic affine testing subspace $A$ has codimension
$
q=2n-2r.
$
Its logarithmic inverse image
$
\Gamma_A=\Log^{-1}(A)
$
is a simple Pfaffian submanifold of $\R^{2n}$ of codimension $q$, dimension $2r$, and coefficient degree at most $2n-1$.
On $\Gamma_A$, the equations
$
\Rea f_j=0
$
and
$
\Ima f_j=0
$
form a square system of $2r$ equations. After the standard toric degree normalization, their degrees are bounded by $2d_j$. Hence their degree product is bounded by
$
\prod_j(2d_j)^2,
$
and their degree-minus-one sum is bounded by
$
4D-2r.
$
Applying the simple-Pfaffian root estimate, together with the permitted generic perturbation, yields
$$
\#\bigl(V\cap\Log^{-1}(A)\bigr)
\le
B_{\mathrm{CI}}^{<}.
$$
The logarithmic map sends this finite lift set surjectively onto
$
A\cap\Log(V).
$
Therefore
$$
\#\bigl(A\cap\Log(V)\bigr)
\le
B_{\mathrm{CI}}^{<}.
$$
Since
$
\CA_V=\Log(V),
$
this proves
$
\#(A\cap\CA_V)
=
\#(A\cap\Log(V))
\le
B_{\mathrm{CI}}^{<}.
$
Finally, the same constant applies to every generic affine testing subspace of dimension
$
2r-n,
$
because it depends only on $n$, $r$, and the degrees $d_j$. Taking the supremum over all such generic affine subspaces gives
$$
\Rdeg_{2n-2r}(\CA_V)
=
\Rdeg_{2n-2r}(\Log(V))
\le
B_{\mathrm{CI}}^{<}.
$$
\end{proof}

\subsection{Why the dimension hypotheses are necessary?}

The expected-dimension hypothesis\\
$
\dim_{\R}\Log(V)=2n-2r
$
ensures that affine subspaces of dimension
$
2r-n
$
are complementary to the top-dimensional amoeba strata. Generic transversality then makes the image intersection finite.
The assumption that $A$ avoids lower-dimensional strata ensures that the counted image points lie on the top-dimensional part of the amoeba and that no positive-dimensional or singular intersection component is introduced by the test.
The explicit assumption that
$
V\cap\Log^{-1}(A)
$
is finite is stronger than finiteness of the image intersection. It rules out the possibility that a single logarithmic image has a positive-dimensional phase fiber inside $V$. This finiteness is required for the zero-counting theorem.

\subsection{The perturbation hypothesis}

The Pfaffian root estimate is stated most cleanly for nondegenerate zeros. If the original lifted system has isolated degenerate zeros, one perturbs its coefficients generically without increasing the degree bounds.
To justify the estimate for the original cardinality, choose pairwise disjoint neighborhoods of the finitely many lifted zeros. The perturbation is required to preserve at least one nearby zero in each neighborhood, counted through the positive local intersection multiplicity. Under this hypothesis, the number of original isolated points does not exceed the total number of perturbed nondegenerate zeros. This is exactly the perturbation assumption included in the theorem.


\subsection{Why no conormal minors appear?}

In the regime
$
2r\le n,
$
the map $\Log|_V$ may have full rank $n$, and the contour is the critical-value hypersurface cut out by logarithmic conormal equations.

In the regime
$
2r>n,
$
the source has real dimension
$
2n-2r<n.
$
The map can never have rank $n$, so every point is critical in the ordinary sense. Therefore the standard contour equals the whole amoeba, and no logarithmic conormal equation is required.
Adding conormal-minor equations would impose the stronger condition
$$
\operatorname{rank}_{\R}(d\Log|_{T_zV})<2n-2r,
$$
which defines the maximal-source-rank-drop locus rather than the standard contour.

\begin{proposition}\label{prop:corrected-concise}
Let
$
V\subset(\C^\ast)^n
$
be a smooth complete intersection of codimension $r$, and assume
$
2r>n.
$
Then every point of $V$ is critical for $\Log|_V$ relative to the target dimension, and therefore
$
\CA_V=\Log(V).
$
The maximal-source-rank-drop contour is
$$
\CA_V^{\mathrm{mr}}
=
\Log\left\{
z\in V:
\rank_{\R}d_z(\Log|_V)<2n-2r
\right\},
$$
and one always has
$
\CA_V^{\mathrm{mr}}\subseteq\CA_V.
$
The inclusion is strict if and only if there exists a full-source-rank point whose logarithmic value is not attained by any maximal-source-rank-drop point.
\end{proposition}

\subsection{Equal-degree specialization}

Assume
$
d_j\le d
$
for every $j$. Then
$
D\le rd
$
and
$$
\prod_{j=1}^r(2d_j)^2
\le
(2d)^{2r}.
$$
Hence
$
B_{\mathrm{CI}}^{\mathrm{amoeba}}
\le
2^{(2n-2r)(2n-2r-1)/2}
(2d)^{2r}
\Theta_{n,r,d}^{\,2n-2r},
$
where
$
\Theta_{n,r,d}
=
4rd-2r
+
(2n-1)(2n-2r)
+
1.
$

\subsection{Codimension-two curves in $(\C^\ast)^3$}

Take
$
n=3
$
and
$
r=2.
$
Then
$
2r=4>3=n,
$
and the amoeba has expected real dimension $2$. It is tested by affine lines.
Here
$
q=2n-2r=2.
$
The universal bound becomes
$$
B_{3,2}^{\mathrm{amoeba}}
=
32d_1^2d_2^2
\left(
4(d_1+d_2)+7
\right)^2.
$$

Indeed,
$
\Theta_{\mathrm{amoeba}}
=
4(d_1+d_2)-4+5\cdot2+1
=
4(d_1+d_2)+7.
$
If
$
d_1=d_2=1,
$
then
$$
B_{3,2}^{\mathrm{amoeba}}
=
32\cdot15^2
=
7200.
$$

\begin{corollary}[Codimension-two curve in $(\C^\ast)^3$]
Let
$
V\subset(\C^\ast)^3
$
be a smooth codimension-two complete-intersection curve defined by equations of degrees $d_1,d_2$. Under the hypotheses of Theorem~\ref{thm:CI-LSS-low},
$$
\Rdeg_2(\Log(V))
\le
32d_1^2d_2^2
\left(
4(d_1+d_2)+7
\right)^2.
$$
\end{corollary}

\subsection{Codimension-three curves in $(\C^\ast)^4$}

Take
$
n=4
$
and
$
r=3.
$
Then
$
2r=6>4,
$
and the amoeba has expected real dimension $2$. It is tested by affine planes.

Now
$
q=2n-2r=2.
$
Hence
$
B_{4,3}^{\mathrm{amoeba}}
=
128d_1^2d_2^2d_3^2(4D+9)^2.
$
Indeed,
$
\Theta_{\mathrm{amoeba}}
=
4D-6+7\cdot2+1
=
4D+9.
$

\subsection{Relation with the Lang--Shapiro--Shustin method}

The proof uses the same structural ingredients as the Lang--Shapiro--Shustin argument. A generic rational testing affine subspace is chosen, its logarithmic inverse image is realized as a simple Pfaffian manifold by one-forms with polynomial coefficients of degree $2n-1$, the defining equations of the variety are restricted to that Pfaffian manifold, and Khovanskii's Pfaffian root estimate is applied with explicit degree data.

The geometric difference is that in the hypersurface-contour regime one adds equations defining logarithmic criticality. In the regime $2r>n$, ordinary criticality is automatic, so the defining equations of the complete intersection alone form the required square system.


\begin{remark}
The bound is universal and uses only total degrees. A sparse mixed-volume estimate based on the exact transformed Newton polytopes may be much smaller for a fixed rational testing direction.
\end{remark}


\section{The Maximal-Source-Rank-Drop Contour in the Regime $2r>n$}

Let
$
V=\{f_1=\cdots=f_r=0\}\subset(\C^\ast)^n
$
be a smooth complete intersection of codimension $r$ and complex dimension
$
m=n-r.
$
Assume
$
2r>n.
$
Then
$
2m=2n-2r<n.
$

For the ordinary target-rank definition, every point of $V$ is critical because the real dimension of the source is smaller than the dimension of the target. Thus the standard contour is the whole amoeba:
$
\CA_V=\Log(V).
$
The maximal-source-rank-drop locus is a different object. It consists of the points at which the logarithmic differential has rank strictly smaller than its maximal possible rank $2m$:
$$
\Sigma_1(V)
=
\left\{
z\in V:
\rank_{\R}\bigl(d\Log|_{T_zV}\bigr)\le 2m-1
\right\}.
$$
Its image
$
\CA_V^{\mathrm{mr}}
=
\Log\bigl(\Sigma_1(V)\bigr)
$
is called the {\em maximal-source-rank-drop contour}.
More generally, for an integer $s\ge1$, define
$$
\Sigma_s(V)
=
\left\{
z\in V:
\dim_{\R}\Ker\bigl(d\Log|_{T_zV}\bigr)\ge s
\right\}.
$$
The exact kernel-dimension stratum is
$
\Sigma_s^\circ(V)
=
\Sigma_s(V)\setminus\Sigma_{s+1}(V).
$

\subsection{The real logarithmic Jacobian}

Define the complex logarithmic Jacobian
$$
A_f(z)
=
\left(
z_i\frac{\partial f_j}{\partial z_i}(z)
\right)_
{\substack{1\le j\le r\\1\le i\le n}}.
$$
Write
$
A_f(z)=U_f(z)+iV_f(z),
$
with $U_f(z)$ and $V_f(z)$ real $r\times n$ matrices, and define the real $2r\times n$ matrix
$\di
A_f(z)_{\R}
=
\begin{pmatrix}
U_f(z)\\
V_f(z)
\end{pmatrix}.
$

\begin{lemma}\label{lem:kernel-log} %
For every $z\in V$,
$
\Ker_{\R}\bigl(d\Log|_{T_zV}\bigr)
\simeq
\Ker_{\R}\bigl(A_f(z)_{\R}\bigr).
$
Consequently,
$$
\Sigma_s(V)
=
\left\{
z\in V:
\rank_{\R}A_f(z)_{\R}\le n-s
\right\}.
$$
\end{lemma}

\begin{proof}
Use logarithmic tangent coordinates
$
\dot z_i=z_i\eta_i.
$
The tangent condition is
$
A_f(z)\eta=0,
$
and
$
d\Log_z(\dot z)=\Rea\eta.
$
A tangent vector belongs to the kernel of $d\Log$ precisely when
$
\eta=iy
$
for some $y\in\R^n$. The tangent condition becomes
$
A_f(z)y=0.
$
Separating real and imaginary parts gives
$
U_f(z)y=0,
\,
V_f(z)y=0,
$
which is equivalent to
$
A_f(z)_{\R}y=0.
$
\end{proof}

Since $2r>n$, a generic real $2r\times n$ matrix has rank $n$. Thus the maximal possible rank of $d\Log|_{T_zV}$ is $2m$, and a source-rank drop occurs exactly when the real logarithmic Jacobian acquires a nonzero kernel.

\subsection{Determinantal strata and their expected dimensions}

Let
$$
\mathcal D_s
=
\left\{
M\in M_{2r,n}(\R):
\rank M\le n-s
\right\}.
$$
Its smooth rank stratum is
$$
\mathcal D_s^\circ
=
\left\{
M\in M_{2r,n}(\R):
\rank M=n-s
\right\}.
$$
The codimension of $\mathcal D_s^\circ$ in $M_{2r,n}(\R)$ is
$
s(2r-n+s).
$
Let
$
\Phi_f:V\longrightarrow M_{2r,n}(\R),
\,
\Phi_f(z)=A_f(z)_{\R}.
$

\begin{proposition}\label{prop:expected-dim} %
Assume that $\Phi_f$ is transverse to every smooth determinantal stratum $\mathcal D_s^\circ$. Then $\Sigma_s^\circ(V)$ is a smooth real-analytic manifold of dimension
$
d_s
=
2m-s(2r-n+s).
$
Since $m=n-r$, this is
$
d_s
=
2n-2r-s(2r-n+s).
$
If $d_s<0$, then $\Sigma_s^\circ(V)$ is empty for a generic complete intersection.
\end{proposition}

\begin{proof}
By Lemma~\ref{lem:kernel-log},
$
\Sigma_s^\circ(V)=\Phi_f^{-1}(\mathcal D_s^\circ).
$
The transversality hypothesis and the Preimage Theorem give
$
\dim_{\R}\Sigma_s^\circ(V)
=
\dim_{\R}V-\operatorname{codim}\mathcal D_s^\circ.
$
Since
$
\dim_{\R}V=2m
$
and
$
\operatorname{codim}\mathcal D_s^\circ=s(2r-n+s),
$
the formula follows.
\end{proof}

For the maximal-source-rank-drop stratum $s=1$,
$
d_1
=
2n-2r-(2r-n+1)
=
3n-4r-1.
$

\subsection{Grassmann kernel incidence}

Let
$
\Gr(s,n)
$
be the Grassmannian of real $s$-planes in $\R^n$. Define the incidence space
$$
\mathcal I_s
=
\left\{
(z,K)\in V\times\Gr(s,n):
K\subset\Ker A_f(z)_{\R}
\right\}.
$$
Over $\Sigma_s^\circ(V)$, the kernel has dimension exactly $s$, so the projection
$
\pi_s:\mathcal I_s\longrightarrow\Sigma_s(V)
$
is one-to-one over $\Sigma_s^\circ(V)$.
Choose a standard affine chart of $\Gr(s,n)$. After choosing a pivot set
$
I\subset\{1,\ldots,n\},
$
$
|I|=s,
$
an $s$-plane is represented by an $n\times s$ matrix
$\di
Y_I(U)
=
\begin{pmatrix}
I_s\\
U
\end{pmatrix},
$
after reordering coordinates, where
$
U
$
has $s(n-s)$ real entries.
The kernel-incidence equations are
$
A_f(z)_{\R}Y_I(U)=0.
$
This gives exactly
$
2rs
$
real polynomial equations.
The Grassmannian has
$
\binom ns
$
standard affine charts.

\subsection{Testing dimension}

Assume
$
d_s\ge0.
$
The image
$
\Log(\Sigma_s^\circ(V))
$
has real dimension at most $d_s$. Under the generic finite-image hypothesis, its top-dimensional part has dimension $d_s$. Its real degree is therefore tested by generic affine subspaces
$
A\subset\R^n
$
of dimension
$
n-d_s.
$
Equivalently, $A$ has codimension
$
d_s.
$
For the maximal-source-rank-drop stratum,
$
n-d_1
=
4r-2n+1.
$

\subsection{The Pfaffian manifold}

Let
$
A\subset\R^n
$
be a generic affine subspace of dimension $n-d_s$. As in the method of Lang, Shapiro, and Shustin,
$
\Log^{-1}(A)
$
is a simple Pfaffian submanifold of $\R^{2n}$ of codimension $d_s$. Its defining Pfaffian one-forms have polynomial coefficients of degree at most
$
\mu=2n-1.
$
In one Grassmann chart, set
$
\Gamma_{A,I}
=
\Log^{-1}(A)\times\R^{s(n-s)}.
$
Then
$
\dim_{\R}\Gamma_{A,I}
=
2n-d_s+s(n-s).
$
Using the formula for $d_s$, one obtains
$
2n-d_s+s(n-s)
=
2r+2rs.
$

On $\Gamma_{A,I}$ impose the $2r$ real defining equations
$
\Rea(f_j)=0,
\,
\Ima(f_j)=0,
\,
1\le j\le r,
$
and the $2rs$ kernel-incidence equations
$
A_f(z)_{\R}Y_I(U)=0.
$
Thus the number of polynomial equations is exactly
$
2r+2rs
=
\dim_{\R}\Gamma_{A,I}.
$
The system is square.

\subsection{Degree bounds}

As in the Lang--Shapiro--Shustin normalization, after coordinate inversions and multiplication by suitable Laurent monomials, each $f_j$ may be assumed to have degree at most
$
2d_j.
$
Hence
$
\deg\Rea(f_j),\deg\Ima(f_j)\le2d_j.
$
The entries of the two real rows of $A_f(z)_{\R}$ associated with $f_j$ have degree at most
$
2d_j.
$
Multiplication by the Grassmann-chart variables increases the total degree by at most one. Therefore each of the $2s$ kernel-incidence equations associated with $f_j$ has degree at most
$
2d_j+1.
$
The product of all polynomial degrees is
$\di
\mathcal P_s
=
\prod_{j=1}^r
(2d_j)^2(2d_j+1)^{2s}.
$
The sum of all degree-minus-one terms is
$\di
\mathcal S_s
=
\sum_{j=1}^r
\left(
2(2d_j-1)+2s((2d_j+1)-1)
\right).
$
Thus
$
\mathcal S_s
=
4(s+1)D-2r,
$
where
$
D=\sum_{j=1}^r d_j.
$
Define
$
\Theta_s
=
4(s+1)D-2r
+
(2n-1)d_s
+
1.
$

\begin{theorem}\label{thm:grassmann-LSS-s} %
Let
$
V\subset(\C^\ast)^n
$
be a smooth complete intersection of codimension $r$ with $2r>n$. Fix $s\ge1$ and assume
$
d_s
=
2n-2r-s(2r-n+s)
\ge0.
$
Assume that the real logarithmic Jacobian map is transverse to the smooth determinantal rank strata. Assume that the top-dimensional part of
$
\Log(\Sigma_s^\circ(V))
$
has dimension $d_s$ and that a generic affine testing subspace
$
A\subset\R^n
$
of dimension $n-d_s$ avoids the images of the deeper strata $\Sigma_{s+1}(V)$.

Assume also that, in every standard Grassmann chart, the resulting square incidence system has only isolated nondegenerate solutions after an arbitrarily small generic perturbation preserving the degrees.
Then
$
\#\left(
A\cap\Log(\Sigma_s^\circ(V))
\right)
\le
B_s^{\mathrm{Gr}},
$
where
$$
B_s^{\mathrm{Gr}}
=
\binom ns
2^{d_s(d_s-1)/2}
\Big(
\prod_{j=1}^r
(2d_j)^2(2d_j+1)^{2s}
\Big)
\Theta_s^{\,d_s}
$$
and
$
\Theta_s
=
4(s+1)D-2r+(2n-1)d_s+1.
$
Consequently,
$
\Rdeg_{d_s}
\Big(
\Log(\Sigma_s^\circ(V))
\Big)
\le
B_s^{\mathrm{Gr}}.
$
\end{theorem}

\begin{proof}
Fix one standard affine chart of $\Gr(s,n)$. The product
$
\Gamma_{A,I}
=
\Log^{-1}(A)\times\R^{s(n-s)}
$
is a simple Pfaffian manifold of codimension $d_s$, dimension $2r+2rs$, and Pfaffian coefficient degree at most $2n-1$.
The $2r$ defining equations and the $2rs$ kernel-incidence equations form a square system on $\Gamma_{A,I}$. Their degree product is $\mathcal P_s$, and their degree-minus-one sum is $\mathcal S_s$.
Khovanskii's Pfaffian root estimate gives, in one Grassmann chart,
$$
2^{d_s(d_s-1)/2}
\mathcal P_s
\left(
\mathcal S_s+(2n-1)d_s+1
\right)^{d_s}.
$$
Substituting the formulas for $\mathcal P_s$ and $\mathcal S_s$ gives the one-chart bound.
There are $ \binom ns$ standard Grassmann charts. Summing over all charts gives $B_s^{\mathrm{Gr}}$.
Every point of
$
A\cap\Log(\Sigma_s^\circ(V))
$
has at least one lift to the Grassmann incidence space. Since $A$ avoids the image of $\Sigma_{s+1}(V)$, the kernel dimension is exactly $s$ at every relevant point, so the Grassmann fiber consists of one $s$-plane. Different charts may represent the same plane, and distinct points of $V$ may have the same logarithmic image, so the number of image points is no larger than the chart-summed number of incidence solutions.
\end{proof}

\subsection{The maximal-source-rank-drop theorem}

Set
$
s=1.
$
Then
$
d_1=3n-4r-1.
$
The testing affine subspaces have dimension
$
n-d_1=4r-2n+1.
$
The degree product is
$
\mathcal P_1
=
\prod_{j=1}^r
(2d_j)^2(2d_j+1)^2,
$
and
$
\Theta_1
=
8D-2r+(2n-1)(3n-4r-1)+1.
$

\begin{theorem}\label{thm:maximal-source-rank-drop}  %
Let
$
V\subset(\C^\ast)^n
$
be a smooth complete intersection of codimension $r$ with
$
2r>n.
$
Assume
$
d_{\mathrm{mr}}
=
3n-4r-1
\ge0.
$
Assume that the real logarithmic Jacobian map is transverse to the smooth rank strata, that
$
\CA_V^{\mathrm{mr}}
=
\Log(\Sigma_1(V))
$
has top-dimensional part of dimension $d_{\mathrm{mr}}$, and that a generic affine subspace
$
A\subset\R^n
$
of dimension
$
4r-2n+1
$
avoids the image of $\Sigma_2(V)$.
Assume that the Grassmann kernel-incidence systems in the $n$ standard charts of $\Gr(1,n)=\mathbb RP^{n-1}$ have only isolated nondegenerate solutions after a generic degree-preserving perturbation.
Then
$$
\#\left(
A\cap\CA_V^{\mathrm{mr}}
\right)
\le
B_{\mathrm{mr}}^{\mathrm{Gr}},
$$
where
$$
B_{\mathrm{mr}}^{\mathrm{Gr}}
=
n\,
2^{d_{\mathrm{mr}}(d_{\mathrm{mr}}-1)/2}
\Big(
\prod_{j=1}^r
(2d_j)^2(2d_j+1)^2
\Big)
\Theta_{\mathrm{mr}}^{\,d_{\mathrm{mr}}}
$$
and
$
\Theta_{\mathrm{mr}}
=
8D-2r
+
(2n-1)d_{\mathrm{mr}}
+
1.
$
Consequently,
$
\Rdeg_{d_{\mathrm{mr}}}
\Big(
\CA_V^{\mathrm{mr}}
\Big)
\le
B_{\mathrm{mr}}^{\mathrm{Gr}}.
$
\end{theorem}

\begin{proof}
This is Theorem~\ref{thm:grassmann-LSS-s} with $s=1$. Since
$
\Gr(1,n)=\mathbb RP^{n-1},
$
there are $n$ standard affine charts. The expected dimension, testing dimension, degree product, and Pfaffian bracket simplify exactly as displayed.
\end{proof}

\begin{corollary}\label{cor:generic-empty}  %
If
$
3n-4r-1<0,
$
then the maximal-source-rank-drop locus is empty for a generic smooth complete intersection satisfying the transversality hypothesis.
\end{corollary}

\begin{proof}
The expected dimension is negative. A transverse inverse image of the corresponding smooth determinantal stratum cannot have negative dimension.
\end{proof}

\subsection{Codimension-two curves in $(\C^\ast)^3$}

Take
$
n=3,
\,
r=2.
$
Then
$
d_{\mathrm{mr}}
=
3\cdot3-4\cdot2-1
=
0.
$
Thus the maximal-source-rank-drop contour is expected to be zero-dimensional. The testing affine subspace is all of $\R^3$.
The bound becomes
$$
B_{\mathrm{mr}}^{\mathrm{Gr}}
=
3
\prod_{j=1}^2
(2d_j)^2(2d_j+1)^2.
$$
There is no Pfaffian exponential or bracket contribution because
$
d_{\mathrm{mr}}=0.
$
If
$
d_1=d_2=1,
$
then
$$
B_{\mathrm{mr}}^{\mathrm{Gr}}
=
3(4\cdot9)^2
=
3888.
$$

\begin{corollary}\label{cor:n3r2}  %
Let
$
V\subset(\C^\ast)^3
$
be a smooth codimension-two complete-intersection curve. Under the hypotheses of Theorem~\ref{thm:maximal-source-rank-drop}, the number of distinct logarithmic images of maximal-source-rank-drop points satisfies
$$
\#\CA_V^{\mathrm{mr}}
\le
3(2d_1)^2(2d_1+1)^2(2d_2)^2(2d_2+1)^2.
$$
\end{corollary}

\subsection{Comparison with the ordinary contour}

In the regime $2r>n$, the ordinary contour is
$
\CA_V=\Log(V),
$
and its degree is tested by affine subspaces of dimension $2r-n$.

The maximal-source-rank-drop contour is usually smaller. Its expected dimension is
$
3n-4r-1,
$
and its degree is tested by affine subspaces of dimension
$
4r-2n+1.
$
The two theorems therefore estimate different geometric objects.

\begin{remark}
The factor $\binom ns$ is a chart-overcount factor. If one Grassmann chart contains all relevant kernel planes, it may be omitted.
\end{remark}

\begin{remark}
The degree bound $2d_j+1$ for the kernel equations is a total-degree majorant. Exact transformed supports can yield much smaller mixed-volume bounds.
\end{remark}

\begin{remark}
If the generic testing affine subspace meets the image of a deeper stratum, the Grassmann fiber becomes positive-dimensional and the incidence system may cease to be zero-dimensional. This is why the theorem explicitly requires avoidance of $\Log(\Sigma_{s+1}(V))$.
\end{remark}


\section*{Appendix A: Complete Proof of the Simple-Pfaffian Root Estimate}


%

The estimate below is the simple-Pfaffian B\'ezout theorem of Khovanskii specialized to polynomial equations on a simple Pfaffian submanifold. The proof makes this specialization explicit and then proves the perturbative extension.

\subsection*{The simple Pfaffian structure}

A simple Pfaffian submanifold of codimension $q$ is given by a chain
$
\R^M=\Gamma_0\supset\Gamma_1\supset\cdots\supset\Gamma_q=\Gamma
$
and polynomial one-forms
$
\omega_1,\ldots,\omega_q
$
such that $\Gamma_j$ is a separating integral hypersurface of the restriction of $\omega_j$ to $\Gamma_{j-1}$. Write
$$
\omega_j
=
\sum_{\nu=1}^M a_{j\nu}(x)\,dx_\nu,
$$
where
$
\deg a_{j\nu}\leq\mu.
$
Because every inclusion has codimension one,
$
\dim_{\R}\Gamma=M-q=k.
$
A common zero $x\in\Gamma$ of the restrictions of
$
P_1,\ldots,P_k
$
is nondegenerate when
$
d(P_1|_\Gamma)_x,\ldots,d(P_k|_\Gamma)_x
$
are linearly independent in $T_x^\ast\Gamma$. Since the number of equations equals $\dim\Gamma$, every nondegenerate zero is isolated.

\subsection*{Khovanskii's simple-Pfaffian B\'ezout theorem}

The foundational result is the following.

\begin{theorem}[Khovanskii]\label{thm:khovanskii}
Let
$
\R^M=\Gamma_0\supset\Gamma_1\supset\cdots\supset\Gamma_q
$
be a simple Pfaffian chain defined by polynomial one-forms whose coefficient degrees are at most $\mu$, and put
$
k=M-q.
$
Let $P_1,\ldots,P_k$ be real polynomials of degrees $p_1,\ldots,p_k$. If their restrictions to $\Gamma_q$ have only isolated nondegenerate common zeros, then their number is at most
$$
2^{q(q-1)/2}
\left(
\prod_{\ell=1}^k p_\ell
\right)
\left(
\sum_{\ell=1}^k(p_\ell-1)+\mu q+1
\right)^q.
$$
\end{theorem}

The theorem is proved by the Rolle--Khovanskii elimination of the $q$ Pfaffian equations. Successively replacing a separating integral hypersurface by a tangency equation reduces the problem to polynomial zero-dimensional systems in the ambient affine space. The successive doubling factors are
$
1,2,2^2,\ldots,2^{q-1},
$
whose product is
$
2^{q(q-1)/2}.
$
The ordinary B\'ezout theorem gives the factor
$\di
\prod_{\ell=1}^k p_\ell.
$
The degree of every tangency polynomial is bounded by
$\di
\sum_{\ell=1}^k(p_\ell-1)+\mu q+1,
$
because differentiating $P_\ell$ contributes $p_\ell-1$, while the polynomial coefficients of the $q$ one-forms contribute at most $\mu q$. The $q$ eliminated Pfaffian equations produce the $q$-th power of this common degree majorant.
The complete Rolle--Khovanskii induction, including the construction of the tangency determinants and the separating-solution argument, is given in Khovanskii's \emph{Fewnomials}, Section~3.12.  

\subsection*{The nondegenerate case}

Let
$
Z
=
\{x\in\Gamma:
P_1(x)=\cdots=P_k(x)=0\}.
$
By hypothesis every point of $Z$ is nondegenerate. All the assumptions of Theorem~\ref{thm:khovanskii} are satisfied with
$
M-k=q,
\,
\deg P_\ell=p_\ell,
\, 
\deg\omega_j\leq\mu.
$
Therefore
$$
\#Z
\leq
2^{q(q-1)/2}
\Big(
\prod_{\ell=1}^k p_\ell
\Big)
\Big(
\sum_{\ell=1}^k(p_\ell-1)+\mu q+1
\Big)^q.
$$
This proves the first assertion.

\subsection*{The perturbative extension}

Assume now that $Z$ is finite but may contain degenerate zeros. Let
$
P_{\ell,\varepsilon}
=
P_\ell+\varepsilon Q_\ell,
$
where the $Q_\ell$ are chosen so that
$
\deg P_{\ell,\varepsilon}\leq p_\ell.
$
Assume that the perturbation is generic, all common zeros of the perturbed restrictions on $\Gamma$ are nondegenerate, and every original isolated zero contributes at least one nearby perturbed zero counted with positive multiplicity.

Since $Z$ is finite, choose pairwise disjoint relatively compact coordinate neighborhoods
$
U_x\subset\Gamma,
\, 
x\in Z,
$
such that $x$ is the only original common zero in $\overline{U_x}$.
For sufficiently small generic $\varepsilon$, the persistence hypothesis gives at least one perturbed zero in every $U_x$. Since the neighborhoods are disjoint, these perturbed zeros are distinct. If
$$
Z_\varepsilon
=
\{y\in\Gamma:
P_{1,\varepsilon}(y)=\cdots=P_{k,\varepsilon}(y)=0\},
$$
then
$
\#Z\leq\#Z_\varepsilon.
$
The perturbed system is nondegenerate and degree preserving. Applying the already proved nondegenerate estimate gives
$$
\#Z_\varepsilon
\leq
2^{q(q-1)/2}
\left(
\prod_{\ell=1}^k p_\ell
\right)
\left(
\sum_{\ell=1}^k(p_\ell-1)+\mu q+1
\right)^q.
$$
Hence the same bound holds for $\#Z$.

\subsection*{Proof of Lemma~\ref{lem:pfaffian}}

\begin{proof}
The nondegenerate assertion is Theorem~\ref{thm:khovanskii} applied to the simple Pfaffian chain defining $\Gamma$. The numerical parameters are
$
q=M-k,
$
$
\deg P_\ell=p_\ell,
$
and
$
\deg\omega_j\leq\mu.
$
Substitution gives the displayed estimate.

For a finite degenerate zero set, use the degree-preserving generic perturbation from the statement. Pairwise disjoint neighborhoods of the original zeros contain distinct perturbed zeros. Therefore the number of original zeros is at most the number of perturbed nondegenerate zeros, which is bounded by the same expression.
\end{proof}

\subsection*{Why the persistence hypothesis is necessary}

An isolated degenerate real zero need not survive every small real perturbation. For example,
$
P(x)=x^2
$
has an isolated zero at $0$, whereas
$
P_\varepsilon(x)=x^2+\varepsilon
$
has no real zero when $\varepsilon>0$.
Thus the perturbative conclusion requires exactly the persistence hypothesis stated in the lemma. It is automatic under additional local hypotheses, such as nonzero local topological degree, but it is not automatic for an arbitrary isolated degenerate real zero.

\begin{remark}
Different conventions for the degree of a polynomial one-form can change the final additive constant in the degree majorant. The convention used here is that every coefficient of every defining one-form has degree at most $\mu$.
\end{remark}


\begin{proof}[Application of Lemma~\ref{lem:pfaffian}]
The ambient real space is
$
\R^{2n},
$
so
$
M=2n.
$
The affine subspace $A$ has codimension
$
q=2n-2r.
$
Hence
$
\Gamma_A=\Log^{-1}(A)
$
is a simple Pfaffian submanifold of dimension
$
k=M-q=2r.
$
The defining polynomial one-forms have coefficient degree at most
$
\mu=2n-1.
$

On $\Gamma_A$, use the $k=2r$ polynomials
$
P_{2j-1}=\Rea f_j
$
and
$
P_{2j}=\Ima f_j.
$
After the toric normalization,
$
p_{2j-1},p_{2j}\le2d_j.
$
Therefore
$$
\prod_{\ell=1}^{2r}p_\ell
\le
\prod_{j=1}^r(2d_j)^2
$$
and
$\di
\sum_{\ell=1}^{2r}(p_\ell-1)
\le
4D-2r.
$
Applying Lemma~\ref{lem:pfaffian} with
$
M=2n,
$
$
q=2n-2r,
$
$
k=2r,
$
and
$
\mu=2n-1
$
gives
$$
\#\bigl(V\cap\Log^{-1}(A)\bigr)
\le
2^{(2n-2r)(2n-2r-1)/2}
\left(
\prod_{j=1}^r(2d_j)^2
\right)
\left(
4D-2r+(2n-1)(2n-2r)+1
\right)^{2n-2r}.
$$
The right-hand side is precisely
$
B_{\mathrm{CI}}^{<}.
$
\end{proof}


\section*{Appendix B: Local Complete-Intersection Description of the Logarithmic Critical Locus,  Proof of Proposition \ref{prop:local-minors}}

Let
$
V=\{f_1=\cdots=f_r=0\}\subset(\C^\ast)^n
$
be a smooth complete intersection of codimension $r$, and assume
$
n\geq2r.
$
For every
$
1\leq j\leq r
$
and
$
1\leq i\leq n,
$
put
$
g_{ji}(z)=z_i\partial f_j/\partial z_i.
$
Write
$
g_{ji}=u_{ji}+iv_{ji},
$
where $u_{ji}$ and $v_{ji}$ are real-valued functions on $(\C^\ast)^n$.

The real logarithmic conormal matrix is
$$
\mathcal M_f(z)
=
\begin{pmatrix}
v_{11}&\cdots&v_{r1}&u_{11}&\cdots&u_{r1}\\
\vdots&&\vdots&\vdots&&\vdots\\
v_{1n}&\cdots&v_{rn}&u_{1n}&\cdots&u_{rn}
\end{pmatrix},
$$
an $n\times2r$ real matrix.
The critical locus of
$
\Log|_V
$
is characterized by
$$
\Crit(\Log|_V)
=
\left\{
z\in V:
\rank_{\R}\mathcal M_f(z)\leq2r-1
\right\}.
$$

Fix a set of row indices
$
I=\{i_1,\ldots,i_{2r-1}\}\subset\{1,\ldots,n\}
$
and a set of column indices
$$
J=\{j_1,\ldots,j_{2r-1}\}\subset\{1,\ldots,2r\}.
$$
Let
$
A_{I,J}(z)
=
\mathcal M_f(z)_{I,J}
$
be the corresponding
$
(2r-1)\times(2r-1)
$
submatrix. Consider the open set
$
U_{I,J}
=
\left\{
z\in V:
\det A_{I,J}(z)\neq0
\right\}.
$
Since there are $2r$ columns altogether, there is exactly one column not contained in $J$. Denote this nonpivot column by
$
j_\ast.
$
The rows not contained in $I$ are
$
k_1,\ldots,k_c,
$
where
$
c=n-(2r-1)=n-2r+1.
$

For every
$
1\leq\mu\leq c,
$
let
$
\Delta_{I,\mu}(z)
$
be the determinant of the $2r\times2r$ submatrix of $\mathcal M_f(z)$ obtained by taking the pivot rows $I$, adjoining the nonpivot row $k_\mu$, taking the pivot columns $J$, and adjoining the unique nonpivot column $j_\ast$. Thus
$
\Delta_{I,\mu}(z)
=
\det
\mathcal M_f(z)_{I\cup\{k_\mu\},\,J\cup\{j_\ast\}},
$
up to the sign determined by the chosen ordering of rows and columns.

\begin{proposition}[Local complete-intersection description]\label{prop:local-minors}
On the open set where the chosen pivot minor does not vanish, the critical locus of $\Log|_V$ is cut out inside $V$ by the $c=n-2r+1$ equations
$$
\Delta_{I,1}(z)=\cdots=\Delta_{I,c}(z)=0.
$$
Equivalently,
$
\Crit(\Log|_V)\cap U_{I,J}
=
\left\{
z\in U_{I,J}:
\Delta_{I,1}(z)=\cdots=\Delta_{I,c}(z)=0
\right\}.
$
\end{proposition}

\subsection*{Block decomposition of the logarithmic conormal matrix}

After permuting rows and columns, which changes the relevant determinants only by signs, the matrix may be written in block form as
$\di
\mathcal M_f(z)
=
\begin{pmatrix}
A(z)&B(z)\\
C(z)&D(z)
\end{pmatrix}.
$
Here
$
A=A_{I,J}
$
has size
$
(2r-1)\times(2r-1),
$
$
B
$
has size
$
(2r-1)\times1,
$
$
C
$
has size
$
c\times(2r-1),
$
and
$
D
$
has size
$
c\times1.
$
Write
$
C
=
\begin{pmatrix}
C_1\\
\vdots\\
C_c
\end{pmatrix},
\,
D
=
\begin{pmatrix}
D_1\\
\vdots\\
D_c
\end{pmatrix},
$
where each $C_\mu$ is a row vector of length $2r-1$ and each $D_\mu$ is a scalar.
On
$
U_{I,J},
$
the matrix $A$ is invertible. Consequently, the rank of the full matrix can be computed by block Gaussian elimination.

\begin{lemma}[Schur-complement rank formula]\label{lem:schur-rank}
If $A$ is invertible, then
$$
\rank
\begin{pmatrix}
A&B\\
C&D
\end{pmatrix}
=
2r-1+\rank\left(D-CA^{-1}B\right).
$$
\end{lemma}

\begin{proof}
Consider the invertible block matrices
$
L
=
\begin{pmatrix}
I_{2r-1}&0\\
-CA^{-1}&I_c
\end{pmatrix}
$
and
$
R
=
\begin{pmatrix}
I_{2r-1}&-A^{-1}B\\
0&1
\end{pmatrix}.
$
A direct multiplication gives
$$
L
\begin{pmatrix}
A&B\\
C&D
\end{pmatrix}
R
=
\begin{pmatrix}
A&0\\
0&D-CA^{-1}B
\end{pmatrix}.
$$
Multiplication on the left and right by invertible matrices preserves rank. Hence
$$
\rank
\begin{pmatrix}
A&B\\
C&D
\end{pmatrix}
=
\rank(A)+\rank(D-CA^{-1}B).
$$
Since
$
\rank(A)=2r-1,
$
the claimed formula follows.
\end{proof}

\subsection*{Equivalence with the criticality condition}

On the pivot chart, the block $A$ is invertible, so
$
\rank_{\R}\mathcal M_f(z)\geq2r-1.
$
Therefore the criticality condition
$
\rank_{\R}\mathcal M_f(z)\leq2r-1
$
is equivalent to the equality
$
\rank_{\R}\mathcal M_f(z)=2r-1.
$
By Lemma~\ref{lem:schur-rank}, this is equivalent to
$
\rank\left(D-CA^{-1}B\right)=0.
$
Since
$
D-CA^{-1}B
$
is a $c\times1$ column vector, its rank is zero if and only if every entry vanishes. Thus the criticality condition is equivalent to
$
D_\mu-C_\mu A^{-1}B=0,
\,
1\leq\mu\leq c.
$
These are the rational Schur-complement equations. Multiplying the $\mu$-th equation by the nonzero scalar $\det A$ gives
$
(\det A)D_\mu-C_\mu\operatorname{adj}(A)B=0,
$
because
$
A^{-1}
=
\frac{\operatorname{adj}(A)}{\det A}.
$
On the open set where $\det A\neq0$, the rational and polynomial equations are equivalent.

\subsection*{Identification with the maximal minors}

For every $\mu$, consider the $2r\times2r$ matrix
$
M_\mu
=
\begin{pmatrix}
A&B\\
C_\mu&D_\mu
\end{pmatrix}.
$
By the block determinant formula,
$
\det M_\mu
=
\det(A)
\left(
D_\mu-C_\mu A^{-1}B
\right).
$
Equivalently,
$$
\det M_\mu
=
(\det A)D_\mu-C_\mu\operatorname{adj}(A)B.
$$

Up to the sign caused by the row and column permutations used to place the pivot block first, this determinant is precisely
$
\Delta_{I,\mu}.
$
Therefore
$
\Delta_{I,\mu}=0
$
if and only if
$
D_\mu-C_\mu A^{-1}B=0
$
on
$
U_{I,J}.
$
It follows that
$
\rank_{\R}\mathcal M_f(z)\leq2r-1
$
if and only if
$
\Delta_{I,1}(z)=\cdots=\Delta_{I,c}(z)=0.
$

\subsection*{Proof of Proposition~\ref{prop:local-minors}}

\begin{proof}
Let
$
z\in U_{I,J}.
$
Since the pivot block
$
A_{I,J}(z)
$
is invertible,
$
\rank_{\R}\mathcal M_f(z)\geq2r-1.
$
The point $z$ is critical for $\Log|_V$ if and only if
$
\rank_{\R}\mathcal M_f(z)\leq2r-1,
$
and therefore, on the pivot chart, if and only if
$
\rank_{\R}\mathcal M_f(z)=2r-1.
$
Writing $\mathcal M_f$ in block form and applying the Schur-complement rank formula gives
$
\rank_{\R}\mathcal M_f(z)
=
2r-1+
\rank_{\R}
\left(
D(z)-C(z)A(z)^{-1}B(z)
\right).
$
Thus
$
\rank_{\R}\mathcal M_f(z)=2r-1
$
if and only if
$
D(z)-C(z)A(z)^{-1}B(z)=0.
$
This vector equation is equivalent to the $c$ scalar equations
$
D_\mu(z)-C_\mu(z)A(z)^{-1}B(z)=0,
\,
1\leq\mu\leq c.
$
Since
$
\det A(z)\neq0,
$
each equation is equivalent to
$$
(\det A(z))D_\mu(z)
-
C_\mu(z)\operatorname{adj}(A(z))B(z)
=
0.
$$
By the block determinant formula, the left-hand side is, up to sign,
$
\Delta_{I,\mu}(z).
$
Therefore
$$
z\in\Crit(\Log|_V)\cap U_{I,J}
$$
if and only if
$
\Delta_{I,1}(z)=\cdots=\Delta_{I,c}(z)=0.
$
This proves the asserted local description.
\end{proof}

\subsection*{Why exactly $c=n-2r+1$ equations occur}

The pivot block uses
$
2r-1
$
of the $n$ rows. Hence the number of remaining rows is
$
n-(2r-1)=n-2r+1.
$
There is only one nonpivot column because the matrix has $2r$ columns and the pivot block uses $2r-1$ columns.
Every nonpivot row therefore produces exactly one enlarged $2r\times2r$ minor by adjoining that row and the unique nonpivot column to the pivot block. Hence the number of local equations is exactly
$
c=n-2r+1.
$

\subsection*{Local complete-intersection interpretation}

The proposition gives a set-theoretic local description of the critical locus. To conclude that the critical locus is a smooth complete intersection of codimension $c$ inside $V$, one needs an additional transversality or Jacobian-rank hypothesis:
$
d\Delta_{I,1},\ldots,d\Delta_{I,c}
$
must be linearly independent on the relevant locus after restriction to $T_zV$.
Under this additional hypothesis, the implicit function theorem gives
$
\operatorname{codim}_V
\left(
\Crit(\Log|_V)\cap U_{I,J}
\right)
=
c.
$
Since
$
\dim_{\R}V=2n-2r,
$
the corresponding real dimension is
$
2n-2r-c
=
2n-2r-(n-2r+1)
=
n-1.
$
This is the expected dimension of the top critical locus whose logarithmic image is a hypersurface in $\R^n$.
Without the differential-independence assumption, the proposition remains correct as a local equation-theoretic statement, but the word ``complete intersection'' should be interpreted as referring to the number of displayed equations rather than automatically asserting regularity.


\section*{Appendix C: Explicit Parameter Substitution in the Regime $n\geq 2r$}

Let
$
V=\{f_1=\cdots=f_r=0\}\subset(\C^\ast)^n
$
be a smooth complete intersection of codimension $r$, and assume
$
n\geq2r.
$
Write
$
D=\sum_{j=1}^r d_j.
$
Let
$
L\subset\R^n
$
be a generic affine line transverse to the top-dimensional smooth strata of the contour
$
\CA_V.
$

The purpose of this section is to apply the simple-Pfaffian root estimate to the critical-lift system over $L$, identify every parameter
$
M,
$
$
q,
$
$
k,
$
$
\mu,
$
and every polynomial degree
$
p_\ell,
$
and derive both the one-chart constant and the global pivot-chart factor without any hidden substitution.

\begin{lemma}[Simple-Pfaffian root estimate]\label{lem:pfaffian2}
Let $\Gamma\subset\R^M$ be a simple Pfaffian submanifold of codimension $q$ defined by polynomial one-forms whose coefficient degrees are at most $\mu$. Assume
$
\dim_{\R}\Gamma=k.
$
Let
$
P_1,\ldots,P_k
$
be real polynomials of degrees
$
p_1,\ldots,p_k
$
whose restrictions to $\Gamma$ have only isolated nondegenerate common zeros. Then the number of those zeros is at most
$$
2^{q(q-1)/2}
\Big(
\prod_{\ell=1}^k p_\ell
\Big)
\Big(
\sum_{\ell=1}^k(p_\ell-1)+\mu q+1
\Big)^q.
$$
\end{lemma}

Write
$
z_i=x_i+iy_i.
$
The complex torus
$
(\C^\ast)^n
$
is an open subset of
$
\R^{2n}
$
with coordinates\\ 
 $
x_1,y_1,\ldots,x_n,y_n.
$
Hence the ambient real dimension in Lemma~\ref{lem:pfaffian} is
$
M=2n.
$

\subsection*{The Pfaffian codimension $q$}

An affine line in $\R^n$ has codimension
$
n-1.
$
Choose independent affine equations
$
\ell_\nu(u)=c_\nu,
\,
1\leq\nu\leq n-1,
$
whose common solution set is $L$.
Pulling these equations back by $\Log$ gives
$
\varphi_\nu(z)
=
\ell_\nu(\Log z)-c_\nu.
$
Therefore
$$
\Gamma_L:=\Log^{-1}(L)
=
\{z\in(\C^\ast)^n:
\varphi_1(z)=\cdots=\varphi_{n-1}(z)=0\}.
$$

Since $\Log$ is a submersion and the affine equations of $L$ are independent, the differentials
$
d\varphi_1,\ldots,$\\
$d\varphi_{n-1}
$
are linearly independent. Thus
$
\Gamma_L
$
has codimension
$
n-1
$
in
$
\R^{2n}.
$
Consequently,
$
q=n-1.
$
The dimension of the simple Pfaffian manifold is
$
k=M-q.
$
Substituting
$
M=2n
$
and
$
q=n-1
$
gives
$
k=2n-(n-1)=n+1.
$
Hence the polynomial system imposed on
$
\Gamma_L
$
must contain exactly
$
n+1
$
equations in order to be square.

\subsection*{The Pfaffian coefficient degree $\mu$}

Write
$\di
\varphi_\nu(z)
=
\sum_{i=1}^n b_{\nu i}\log|z_i|-c_\nu.
$
Then
$$
d\varphi_\nu
=
\sum_{i=1}^n
b_{\nu i}
\frac{x_i\,dx_i+y_i\,dy_i}{x_i^2+y_i^2}.
$$
Multiplying by
$\di
R(x,y)=\prod_{j=1}^n(x_j^2+y_j^2)
$
gives the polynomial one-form
$$
\alpha_\nu
=
\sum_{i=1}^n
b_{\nu i}
(x_i\,dx_i+y_i\,dy_i)
\prod_{j\ne i}(x_j^2+y_j^2).
$$
Every coefficient of $\alpha_\nu$ has degree
$
1+2(n-1)=2n-1.
$
Therefore
$
\mu=2n-1.
$

\subsection*{The logarithmic conormal equations}

Let
$
\mathcal M_f(z)
$
be the real logarithmic conormal matrix of size
$
n\times2r.
$
On the top critical stratum one has
$
\rank_{\R}\mathcal M_f(z)=2r-1.
$
Choose a pivot chart
$
\kappa=(I,J)
$
such that the corresponding
$
(2r-1)\times(2r-1)
$
minor is nonzero. After permuting rows and columns, write
$$
\mathcal M_f
=
\begin{pmatrix}
A_\kappa&B_\kappa\\
C_\kappa&D_\kappa
\end{pmatrix},
$$
where
$
A_\kappa
$
has size
$
(2r-1)\times(2r-1).
$

The criticality condition is equivalent on this chart to
$
D_\kappa-C_\kappa A_\kappa^{-1}B_\kappa=0.
$
After clearing the denominator, one obtains
$
S_{\kappa,\mu}
=
(\det A_\kappa)D_{\kappa,\mu}
-
C_{\kappa,\mu}\operatorname{adj}(A_\kappa)B_\kappa
=
0.
$
The number of nonpivot rows is
$
n-(2r-1)=n-2r+1.
$
Put
$
c=n-2r+1.
$
Hence there are exactly
$
c
$
Schur-complement equations.

\subsection*{The number of polynomial equations}

The defining equations of $V$ give
$
2r
$
real equations:
$
\Rea f_j=0,
\, 
\Ima f_j=0,
\, 
1\leq j\leq r.
$
The criticality condition gives
$
c=n-2r+1
$
additional real equations:
$$
S_{\kappa,1}=0,\ldots,S_{\kappa,c}=0.
$$

Therefore the total number of polynomial equations is
$
2r+c
=
2r+(n-2r+1)
=
n+1.
$
This equals
$
k=\dim_{\R}\Gamma_L.
$
Thus the system is square on
$
\Gamma_L.
$

\subsection*{The degrees of the defining equations}

After the same toric normalization used in the Lang--Shapiro--Shustin method, the real and imaginary parts of $f_j$ have degrees at most
$
2d_j.
$
Define
$
P_{2j-1}=\Rea f_j,
\,
P_{2j}=\Ima f_j.
$
Then
$
p_{2j-1}\leq2d_j,
\,
p_{2j}\leq2d_j.
$
For the uniform degree-only bound, one uses
$
p_{2j-1}=p_{2j}=2d_j.
$

\subsection*{The degrees of the Schur-complement equations}

Each maximal conormal minor uses exactly two columns associated with each defining polynomial $f_j$. Every such column has degree at most
$
2d_j.
$
Therefore every determinant monomial has degree at most
$\di
2\sum_{j=1}^r2d_j=4D.
$
Hence
$
\deg S_{\kappa,\mu}\leq4D,
\, 
1\leq\mu\leq c.
$
For the uniform bound, write
$
p_{2r+\mu}=4D,
\, 
1\leq\mu\leq c.
$
The complete degree list is therefore
$
p_{2j-1}=p_{2j}=2d_j,
\, 
1\leq j\leq r,
$
and
$
p_{2r+\mu}=4D,
\, 
1\leq\mu\leq n-2r+1.
$

\subsection*{The product of the polynomial degrees}

The degree product in Lemma~\ref{lem:pfaffian} is
$\di
\prod_{\ell=1}^{k}p_\ell
=
\prod_{\ell=1}^{n+1}p_\ell.
$
The first
$
2r
$
factors give
$\di
\prod_{j=1}^r(2d_j)^2.
$
The remaining
$
c=n-2r+1
$
factors are all bounded by
$
4D.
$
Thus
$\di
\prod_{\ell=1}^{n+1}p_\ell
\leq
\Big(
\prod_{j=1}^r(2d_j)^2
\Big)
(4D)^{n-2r+1}.
$
This is the polynomial-degree product in the one-chart estimate.

\subsection*{The degree-minus-one sum}

For the defining equations,
$\di
\sum_{\ell=1}^{2r}(p_\ell-1)
=
\sum_{j=1}^r
\left(
(2d_j-1)+(2d_j-1)
\right).
$

Hence
$\di
\sum_{\ell=1}^{2r}(p_\ell-1)
=
4D-2r.
$
For the
$
c=n-2r+1
$
Schur-complement equations,
$\di
\sum_{\mu=1}^{c}
(p_{2r+\mu}-1)
=
c(4D-1).
$
Therefore
$\di
\sum_{\ell=1}^{n+1}(p_\ell-1)
=
4D-2r+(n-2r+1)(4D-1).
$

\subsection*{The Pfaffian contribution $\mu q$}

The coefficient-degree parameter is
$
\mu=2n-1,
$
and the Pfaffian codimension is
$
q=n-1.
$
Hence
$
\mu q
=
(2n-1)(n-1).
$
The full bracket in Lemma~\ref{lem:pfaffian} is therefore
$\di
\sum_{\ell=1}^{n+1}(p_\ell-1)+\mu q+1
=
4D-2r
+
(n-2r+1)(4D-1)
+
(2n-1)(n-1)
+
1.
$
Define
$$
\Theta_{\mathrm{CI}}
=
4D-2r
+
(n-2r+1)(4D-1)
+
(2n-1)(n-1)
+
1.
$$
Thus the bracket is exactly
$
\Theta_{\mathrm{CI}}.
$
The exponential Pfaffian factor is
$
2^{q(q-1)/2}.
$
Since
$
q=n-1,
$
one obtains
$
2^{q(q-1)/2}
=
2^{(n-1)(n-2)/2}.
$

{\it Complete substitution into Lemma~\ref{lem:pfaffian2}.}
Every parameter is now explicit:
$
M=2n,
\,
q=n-1,
\,
k=n+1,
\,
\mu=2n-1,
$
$
c=n-2r+1,
$\, 
$
p_{2j-1}=p_{2j}=2d_j,
\,
1\leq j\leq r,
$
and
$
p_{2r+\mu}=4D,
\, 
1\leq\mu\leq c.
$

Lemma~\ref{lem:pfaffian2} gives, in the pivot chart $\kappa$,
$$
\#Z_{L,\kappa}
\leq
2^{(n-1)(n-2)/2}
\Big(
\prod_{j=1}^r(2d_j)^2
\Big)
(4D)^{n-2r+1}
\Theta_{\mathrm{CI}}^{\,n-1},
$$
where
$
\Theta_{\mathrm{CI}}
=
4D-2r
+
(n-2r+1)(4D-1)
+
(2n-1)(n-1)
+
1.
$
Define
$$
B_{\mathrm{CI}}^{\mathrm{one}}
=
2^{(n-1)(n-2)/2}
\Big(
\prod_{j=1}^r(2d_j)^2
\Big)
(4D)^{n-2r+1}
\Theta_{\mathrm{CI}}^{\,n-1}.
$$
Then
$
\#Z_{L,\kappa}
\leq
B_{\mathrm{CI}}^{\mathrm{one}}.
$

\subsection*{From critical lifts to contour points}

Every point
$
x\in L\cap\CA_V
$
has at least one critical lift
$
z\in V
$
such that
$
\Log(z)=x.
$
If all relevant lifts lie in the chosen chart $\kappa$, then
$$
\#(L\cap\CA_V)
\leq
\#Z_{L,\kappa}
\leq
B_{\mathrm{CI}}^{\mathrm{one}}.
$$
Distinct critical lifts may have the same logarithmic image, so the passage from lifts to contour points cannot increase the count.

\subsection*{The global pivot-chart factor}

The top critical stratum consists of matrices of rank exactly
$
2r-1.
$
Every such matrix has at least one nonzero
$
(2r-1)\times(2r-1)
$
minor.
The number of choices of
$
2r-1
$
rows among $n$ rows is
$
\binom{n}{2r-1}.
$
The number of choices of
$
2r-1
$
columns among $2r$ columns is
$
\binom{2r}{2r-1}=2r.
$
Hence the number of pivot charts is
$
N_{\mathrm{chart}}
=
2r\binom{n}{2r-1}.
$

Let
$
Z_{L,\kappa}
$
be the set of critical lifts in chart $\kappa$. Since the pivot charts cover the top critical stratum,
$\di
Z_L
\subseteq
\bigcup_\kappa Z_{L,\kappa}.
$
Therefore
$\di
\#Z_L
\leq
\sum_\kappa\#Z_{L,\kappa}.
$
Applying the same one-chart bound in every chart gives
$\di
\#Z_L
\leq
2r\binom{n}{2r-1}
B_{\mathrm{CI}}^{\mathrm{one}}.
$
Define
$\di
B_{\mathrm{CI}}^{\mathrm{global}}
=
2r\binom{n}{2r-1}
B_{\mathrm{CI}}^{\mathrm{one}}.
$
Thus
$$
B_{\mathrm{CI}}^{\mathrm{global}}
=
2r\binom{n}{2r-1}
2^{(n-1)(n-2)/2}
\Big(
\prod_{j=1}^r(2d_j)^2
\Big)
(4D)^{n-2r+1}
\Theta_{\mathrm{CI}}^{\,n-1}.
$$
Since
$
\#(L\cap\CA_V)\leq\#Z_L,
$
one obtains
$
\#(L\cap\CA_V)
\leq
B_{\mathrm{CI}}^{\mathrm{global}}.
$
Taking the supremum over all generic affine lines gives
$$
\Rdeg(\CA_V)
\leq
B_{\mathrm{CI}}^{\mathrm{global}}.
$$

\begin{proof}[Application of Lemma~\ref{lem:pfaffian} in the regime $n\geq2r$]
The ambient real space is
$
\R^{2n},
$
so
$
M=2n.
$
Since $L$ is an affine line, $\Gamma_L=\Log^{-1}(L)$ has Pfaffian codimension
$
q=n-1
$
and dimension
$
k=M-q=n+1.
$
The polynomial one-forms defining $\Gamma_L$ have coefficient degree at most
$
\mu=2n-1.
$

On one pivot chart, the square system consists of the $2r$ equations
$
\Rea f_j=0
$
and
$
\Ima f_j=0,
$
together with
$
c=n-2r+1
$
Schur-complement equations. After toric normalization, their degrees satisfy
$
p_{2j-1},p_{2j}\leq2d_j
$
and
$
p_{2r+\mu}\leq4D.
$
Hence
$\di
\prod_{\ell=1}^{n+1}p_\ell
\leq
\Big(
\prod_{j=1}^r(2d_j)^2
\Big)
(4D)^{n-2r+1}
$
and
$\di
\sum_{\ell=1}^{n+1}(p_\ell-1)
\leq
4D-2r+(n-2r+1)(4D-1).
$
Applying Lemma~\ref{lem:pfaffian} with
$
M=2n,
$
$
q=n-1,
$
$
k=n+1,
$
and
$
\mu=2n-1
$
gives
$\di
\#Z_{L,\kappa}
\leq
2^{(n-1)(n-2)/2}
\Big(
\prod_{j=1}^r(2d_j)^2
\Big)
(4D)^{n-2r+1}
\Theta_{\mathrm{CI}}^{\,n-1},
$
where
$
\Theta_{\mathrm{CI}}
=
4D-2r
+
(n-2r+1)(4D-1)
+
(2n-1)(n-1)
+
1.
$
This is the one-chart constant
$
B_{\mathrm{CI}}^{\mathrm{one}}.
$
The rank-$2r-1$ stratum is covered by
$
2r\binom{n}{2r-1}
$
pivot charts. Summing the one-chart estimate gives
$\di
B_{\mathrm{CI}}^{\mathrm{global}}
=
2r\binom{n}{2r-1}
B_{\mathrm{CI}}^{\mathrm{one}}.
$
\end{proof}

\subsection{The equality case $n=2r$}

When
$
n=2r,
$
one has
$
c=n-2r+1=1.
$
Thus there is exactly one Schur-complement equation.
The one-chart constant becomes
$$
B_{\mathrm{CI}}^{\mathrm{eq}}
=
2^{(n-1)(n-2)/2}
\Big(
\prod_{j=1}^r(2d_j)^2
\Big)
(4D)
\Theta_{\mathrm{eq}}^{\,n-1},
$$
where
$
\Theta_{\mathrm{eq}}
=
4D-2r+(4D-1)+(2n-1)(n-1)+1.
$
Since
$
n=2r,
$
this simplifies to
$
\Theta_{\mathrm{eq}}
=
8D-n+(2n-1)(n-1).
$


\section*{Appendix D: $\Log^{-1}(L)$ Is a Simple Pfaffian Submanifold of Codimension $n-1$}

Let
$
L\subset\R^n
$
be an affine line. Since an affine line has real dimension $1$, its codimension in $\R^n$ is
$
\operatorname{codim}_{\R^n}(L)=n-1.
$
Consequently, there are $n-1$ linearly independent affine-linear equations whose common zero set is $L$. Thus there exist vectors
$
b^{(1)},\ldots,b^{(n-1)}\in\R^n
$
and constants
$
c_1,\ldots,c_{n-1}\in\R
$
such that
$$
L
=
\left\{
u\in\R^n:
\langle b^{(\nu)},u\rangle=c_\nu,
\quad
1\le\nu\le n-1
\right\}.
$$

Write
$
z_i=x_i+iy_i.
$
Since
$
|z_i|=\sqrt{x_i^2+y_i^2},
$
one has
$\di
\log|z_i|
=
\frac12\log(x_i^2+y_i^2).
$
A point
$
z\in(\C^\ast)^n
$
belongs to
$
\Log^{-1}(L)
$
if and only if
$\di
\sum_{i=1}^n b_i^{(\nu)}\log|z_i|
=
c_\nu,
\,
1\le\nu\le n-1.
$
Define
$\di
\varphi_\nu(z)
=
\sum_{i=1}^n b_i^{(\nu)}\log|z_i|-c_\nu.
$
Then
$$
\Log^{-1}(L)
=
\left\{
z\in(\C^\ast)^n:
\varphi_1(z)=\cdots=\varphi_{n-1}(z)=0
\right\}.
$$

The differential of $\varphi_\nu$ is
$\di
d\varphi_\nu
=
\sum_{i=1}^n
b_i^{(\nu)}
\frac{x_i\,dx_i+y_i\,dy_i}{x_i^2+y_i^2}.
$
The denominators do not vanish on $(\C^\ast)^n$. Multiplying by
$\di
R(x,y)
=
\prod_{j=1}^n(x_j^2+y_j^2)
$
gives the polynomial one-form
$
\alpha_\nu
=
R(x,y)d\varphi_\nu.
$
Explicitly,
$\di
\alpha_\nu
=
\sum_{i=1}^n
b_i^{(\nu)}
\left(
x_i\,dx_i+y_i\,dy_i
\right)
\prod_{j\ne i}(x_j^2+y_j^2).
$
Every coefficient of $\alpha_\nu$ is a polynomial of degree
$
1+2(n-1)=2n-1.
$
Since $R(x,y)>0$ on $(\C^\ast)^n$, the one-forms $\alpha_\nu$ and $d\varphi_\nu$ have the same kernels. Hence the level hypersurface
$
\{\varphi_\nu=0\}
$
is an integral hypersurface of the polynomial Pfaffian equation
$
\alpha_\nu=0.
$

Define
$
\Gamma_0=(\C^\ast)^n
$
and successively
$
\Gamma_\nu
=
\left\{
z\in\Gamma_{\nu-1}:
\varphi_\nu(z)=0
\right\},
\, 
1\le\nu\le n-1.
$
Then
$
\Gamma_{n-1}=\Log^{-1}(L).
$
This nested construction is what is meant by saying that $\Log^{-1}(L)$ is a simple Pfaffian submanifold: it is obtained by intersecting successive separating integral hypersurfaces of polynomial one-forms.
It remains to justify that the $n-1$ equations are independent. Let
$
B:\R^n\longrightarrow\R^{n-1}
$
be the linear map whose rows are the vectors $b^{(\nu)}$. Since these vectors are linearly independent,
$
\operatorname{rank}B=n-1.
$
Moreover,
$
(\varphi_1,\ldots,\varphi_{n-1})
=
B\circ\Log-c.
$
Therefore
$
d(\varphi_1,\ldots,\varphi_{n-1})_z
=
B\circ d\Log_z.
$

The differential of $\Log$ is surjective. Indeed,
$\di
d\Log_z(\dot x,\dot y)
=
\left(
\frac{x_1\dot x_1+y_1\dot y_1}{x_1^2+y_1^2},
\ldots,
\frac{x_n\dot x_n+y_n\dot y_n}{x_n^2+y_n^2}
\right).
$
Given
$
v=(v_1,\ldots,v_n)\in\R^n,
$
choose
$
\dot x_i=v_ix_i,
\,
\dot y_i=v_iy_i.
$
Then
$
d\Log_z(\dot x,\dot y)=v.
$
Thus
$
\operatorname{rank}d\Log_z=n.
$
Since $B$ has rank $n-1$ and $d\Log_z$ is surjective, the composition
$
B\circ d\Log_z
$
has rank $n-1$. Therefore
$
d\varphi_1,\ldots,d\varphi_{n-1}
$
are linearly independent at every point of $\Log^{-1}(L)$.
The Regular Value Theorem gives
$
\operatorname{codim}_{\R^{2n}}\Log^{-1}(L)
=
n-1.
$
Since
$
\dim_{\R}(\C^\ast)^n=2n,
$
one obtains
$
\dim_{\R}\Log^{-1}(L)
=
2n-(n-1)
=
n+1.
$

\begin{proposition}
Let
$
L\subset\R^n
$
be an affine line. Then
$
\Log^{-1}(L)
$
is a smooth real-analytic submanifold of $(\C^\ast)^n$ of codimension $n-1$ and dimension $n+1$. Moreover, it is a simple Pfaffian submanifold defined by $n-1$ polynomial one-forms whose coefficients have degree at most $2n-1$.
\end{proposition}

\begin{proof}
The smoothness and codimension follow from the surjectivity of $d\Log$ and the independence of the affine equations defining $L$. The Pfaffian description follows from the polynomial one-forms
$\di
\alpha_\nu
=
\Big(
\prod_{j=1}^n(x_j^2+y_j^2)
\Big)
d\varphi_\nu.
$
Their integral hypersurfaces are the level sets $\varphi_\nu=0$, and their coefficient degree is at most $2n-1$.
\end{proof}

\subsection*{Direct parametrization}

Write
$
L=\{a+t\ell:t\in\R\},
$
where
$
a,\ell\in\R^n
$
and
$
\ell\ne0.
$
A point of $\Log^{-1}(L)$ satisfies
$
|z_i|=e^{a_i+t\ell_i}.
$
Hence
$
z_i
=
e^{a_i+t\ell_i}e^{i\theta_i},
\,
\theta_i\in\R/2\pi\Z.
$
Therefore
$
\Log^{-1}(L)
\simeq
\R\times(S^1)^n.
$
This directly gives
$
\dim_{\R}\Log^{-1}(L)=1+n=n+1.
$
The Pfaffian description is needed because it allows one to apply Khovanskii's theorem to polynomial equations restricted to $\Log^{-1}(L)$.

We conclude that the codimension $n-1$ comes from the fact that an affine line in $\R^n$ is defined by $n-1$ independent affine equations. Pulling those equations back by $\Log$ gives $n-1$ independent real-analytic equations on $(\C^\ast)^n$. After clearing the nonvanishing denominators in their differentials, one obtains polynomial Pfaffian one-forms of coefficient degree $2n-1$.
Thus
$
\dim_{\R}\Log^{-1}(L)
=
2n-(n-1)
=
n+1.
$





\begin{thebibliography}{99}

\bibitem{Bergman71}
G.~M.~Bergman,
\emph{The logarithmic limit-set of an algebraic variety},
Transactions of the American Mathematical Society
\textbf{157} (1971), 459--469.

\bibitem{Bernstein75}
D.~N. Bernstein,
The number of roots of a system of equations,
\emph{Functional Analysis and Its Applications} \textbf{9} (1975), no.~3, 183--185.


\bibitem{BieriGroves84}
R.~Bieri and J.~R.~J.~Groves,
\emph{The geometry of the set of characters induced by valuations},
Journal f\"ur die Reine und Angewandte Mathematik
\textbf{347} (1984), 168--195.







\bibitem{ForsbergPassareTsikh00}
M.~Forsberg, M.~Passare, and A.~Tsikh,
Laurent determinants and arrangements of hyperplane amoebas,
\emph{Advances in Mathematics} \textbf{151} (2000), no.~1, 45--70.

\bibitem{Fulton93}
W.~Fulton,
\emph{Introduction to Toric Varieties},
Annals of Mathematics Studies, Vol.~131,
Princeton University Press, Princeton, NJ, 1993.

\bibitem{Fulton98}
W.~Fulton,
\emph{Intersection Theory},
2nd ed.,
Springer, Berlin, 1998.

\bibitem{GKZ94}
I.~M. Gelfand, M.~M. Kapranov, and A.~V. Zelevinsky,
\emph{Discriminants, Resultants, and Multidimensional Determinants},
Birkh\"auser, Boston, 1994.


\bibitem{GriffithsHarris78}
P.~Griffiths and J.~Harris,
\emph{Principles of Algebraic Geometry},
Wiley-Interscience, New York, 1978.




 \bibitem{Khovanskii91}
A.~G. Khovanskii,
\emph{Fewnomials},
Translations of Mathematical Monographs, Vol.~88,
American Mathematical Society, Providence, RI, 1991.


\bibitem{LangShapiroShustin21}
L.~Lang, B.~Shapiro, and E.~Shustin,
On the number of intersection points of the contour of an amoeba with a line,
\emph{Indiana University Mathematics Journal}
\textbf{70} (2021), no.~4, 1335--1353.
 \bigskip
 

\bibitem{MaclaganSturmfels15}
D.~Maclagan and B.~Sturmfels,
\emph{Introduction to Tropical Geometry},
Graduate Studies in Mathematics, Vol.~161,
American Mathematical Society, Providence, RI, 2015.

\bibitem{Mikhalkin00}
G.~Mikhalkin,
Real algebraic curves, the moment map and amoebas,
\emph{Annals of Mathematics} \textbf{151} (2000), no.~1, 309--326.

\bibitem{Mikhalkin04}
G.~Mikhalkin,
Amoebas of algebraic varieties and tropical geometry,
in \emph{Different Faces of Geometry},
International Mathematical Series, Vol.~3,
Kluwer Academic/Plenum Publishers, New York, 2004, pp.~257--300.

\bibitem{NissePassare17}
M.~Nisse and M.~Passare,
Amoebas and Coamoebas of Linear Spaces,
in M.~Andersson, J.~Boman, C.~Kiselman, P.~Kurasov, and R.~Sigurdsson (eds.),
\emph{Analysis Meets Geometry},
Trends in Mathematics,
Birkh\"auser, Cham, 2017, pp.~63--80.

\bibitem{NisseSottile13}
M.~Nisse and F.~Sottile,
The phase limit set of an algebraic variety,
\emph{Algebra \& Number Theory} \textbf{7} (2013), no.~2, 339--352.

\bibitem{NisseSottile22}
M.~Nisse and F.~Sottile,
Describing amoebas,
\emph{Pacific Journal of Mathematics} \textbf{317} (2022), no.~1, 187--205.

\bibitem{PassareRullgard04}
M.~Passare and H.~Rullg{\aa}rd,
Amoebas, Monge--Amp\`ere measures, and triangulations of the Newton polytope,
\emph{Duke Mathematical Journal} \textbf{121} (2004), no.~3, 481--507.

\bibitem{Tevelev07}
J.~Tevelev,
Compactifications of subvarieties of tori,
\emph{American Journal of Mathematics} \textbf{129} (2007), no.~4, 1087--1104.

\end{thebibliography}
 \end{document}